\documentclass[reqno]{amsart}
\usepackage{amsthm, amssymb, amsfonts}
\usepackage{lmodern}
\usepackage{color}
\usepackage{hyperref}
\usepackage[T5]{fontenc}
\usepackage{amscd,amssymb}
\usepackage[v2,cmtip]{xy}
\usepackage{mathrsfs}
\theoremstyle{plain}
\newtheorem{thm}{Theorem}[section]
\newtheorem{prop}[thm]{Proposition}
\newtheorem{lem}[thm]{Lemma}
\newtheorem{con}[thm]{Conjecture}
\newtheorem{corl}[thm]{Corollary}
\theoremstyle{definition}
\newtheorem{defn}[thm]{Definition}

\theoremstyle{plain}
\newtheorem{thms}{Theorem}[subsection]
\newtheorem{props}[thms]{Proposition}
\newtheorem{lems}[thms]{Lemma}

\newtheorem{corls}[thms]{Corollary}
\theoremstyle{definition}
\newtheorem{defns}[thms]{Definition}

\newtheorem{notas}[thms]{Notation}
\numberwithin{equation}{section}

\begin{document}

\title[The generators of the polynomial algebra]
{On a construction for the generators of the polynomial algebra as a module over the Steenrod algebra}

 \author{Nguy\~\ecircumflex n Sum} 

\address{Department of Mathematics, Quy Nh\ohorn n University, 170 An D\uhorn \ohorn ng V\uhorn \ohorn ng, Quy Nh\ohorn n, B\`inh \DJ \d inh, Viet Nam}

\email{nguyensum@qnu.edu.vn}




\begin{abstract}
 Let $P_n$ be the graded polynomial algebra $\mathbb F_2[x_1,x_2,\ldots ,x_n]$ with the degree of each generator $x_i$
being 1, where $\mathbb F_2$ denote the prime field of two elements. 

The Peterson hit problem is to find a minimal generating set for $P_n$ regarded as a module over the  mod-2 Steenrod algebra, $\mathcal{A}$. Equivalently, we want to find a vector space basis for $\mathbb F_2 \otimes_{\mathcal A} P_n$ in each degree $d$. Such a basis may be represented by a list of monomials of degree $d$.

In this paper, we present a construction for the $\mathcal A$-generators of $P_n$ and  prove some properties of it. We also explicitly determine a basis of $\mathbb F_2 \otimes_{\mathcal A} P_n$ for $n = 5$ and the degree $d = 15.2^s - 5$ with $s$ an arbitrary positive integer. 
These results are used to verify Singer's conjecture for the fifth Singer algebraic transfer in respective degree.
\end{abstract}

\maketitle


\section{Introduction}\label{s1} 
\setcounter{equation}{0}

 Denote by $P_n:= \mathbb F_2[x_1,x_2,\ldots ,x_n]$ the polynomial algebra over the field of two elements, $\mathbb F_2$, in $n$ generators $x_1, x_2, \ldots , x_n$, each of degree 1. This algebra arises as the cohomology with coefficients in $\mathbb F_2$ of a classifying space of an elementary abelian 2-group of rank $n$.  Therefore, $P_n$ is a module over the mod-2 Steenrod algebra, $\mathcal A$.   
The action of $\mathcal A$ on $P_n$ is determined by the elementary properties of the Steenrod squares $Sq^i$ and subject to the Cartan formula
$Sq^k(fg) = \sum_{i=0}^kSq^i(f)Sq^{k-i}(g),$
for $f, g \in P_n$ (see Steenrod and Epstein~\cite{st}).

A polynomial $g$ in $P_n$ is called \textit{hit} if it belongs to  $\mathcal{A}^+P_n$, where $\mathcal{A}^+$ is the augmentation ideal of $\mathcal A$. That means $g$ can be written as a finite sum $g = \sum_{j > 0}Sq^{j}(g_j)$ for suitable polynomials $g_j \in P_n$.  

We study the \textit{Peterson hit problem} of determining a minimal set of generators for the polynomial algebra $P_n$ as a module over the Steenrod algebra. Equivalently, we want to find a vector space basis for the quotient 
$$QP_n := P_n/\mathcal A^+P_n = \mathbb F_2 \otimes_{\mathcal A} P_n.$$ 
The problem was first studied by Peterson~\cite{pe}, Wood~\cite{wo}, Singer~\cite {si1}, and Priddy \cite{pr}, who showed its relation to several classical problems in the homotopy theory. Then, this problem  was investigated by Carlisle and Wood~\cite{cw}, Crabb and Hubbuck~\cite{ch}, Janfada and Wood~\cite{jw1}, Kameko~\cite{ka}, Mothebe \cite{mo}, Nam~\cite{na}, Phuc and Sum \cite{sp,sp2}, Silverman~\cite{sl}, Silverman and Singer~\cite{ss}, Singer~\cite{si2}, Walker and Wood~\cite{wa}, Wood~\cite{wo2} and others.

The vector space $QP_n$ was explicitly calculated by Peterson~\cite{pe} for $n=1, 2,$ by Kameko~\cite{ka} for $n=3$, and recently by us \cite{su2,su5}  for $n=4$, unknown in general. 

Let $GL_n$ be the general linear group over the field $\mathbb F_2$. This group acts naturally on $P_n$ by matrix substitution. Since the two actions of $\mathcal A$ and $GL_n$ upon $P_n$ commute with each other, there is an inherited action of $GL_n$ on $QP_n$. 

Denote by $(P_n)_d$ the subspace of $P_n$ consisting of all the homogeneous polynomials of degree $d$ in $P_n$ and by $(QP_n)_d$ the subspace of $QP_n$ consisting of all the classes represented by the elements in $(P_n)_d$. 
In \cite{si1}, Singer defined the algebraic transfer,  which is a homomorphism
$$\varphi_n :\text{Tor}^{\mathcal A}_{n,n+d} (\mathbb F_2,\mathbb F_2) \longrightarrow  (QP_n)_d^{GL_n}$$
from the homology of the Steenrod algebra to the subspace of $(QP_n)_d$ consisting of all the $GL_n$-invariant classes. It is a useful tool in describing the homology groups of the Steenrod algebra, $\text{Tor}^{\mathcal A}_{n,n+d} (\mathbb F_2,\mathbb F_2)$. This transfer was studied by  Boardman~\cite{bo}, Bruner, H\`a and H\uhorn ng~\cite{br}, H\uhorn ng ~\cite{hu}, Ch\ohorn n-H\`a~ \cite{cha1,cha2}, Minami ~\cite{mi}, Nam~ \cite{na2}, H\uhorn ng-Qu\`ynh~ \cite{hq}, the present author \cite{su4} and others.

It was shown that the transfer is an isomorphism for $n = 1, 2$ by Singer in \cite{si1} and for $n = 3$ by Boardman in \cite{bo}. However, for any $n \geqslant 4$,  $\varphi_n$ is not a monomorphism in infinitely many degrees (see Singer \cite{si1}, H\uhorn ng \cite{hu}.) Singer made a conjecture that \textit{the algebraic transfer $\varphi_n$ is an epimorphism for any $n \geqslant 0$.}
This conjecture is true for $n\leqslant 3$. It can be verified for $n=4$ by using the results in \cite{su2,su5}. The conjecture  for $n\geqslant 5$ is an open problem.

In this paper, we present a construction for the $\mathcal A$-generators of $P_n$ and explicitly determine a basis of $\mathbb F_2 \otimes_{\mathcal A} P_n$ for $n = 5$ and the degree $d = 15.2^s - 5$ with $s$ an arbitrary positive integer. These results are used to study the fifth Singer algebraic transfer. We will show that Singer's conjecture for the fifth algebraic transfer is true in this case.

\medskip
This paper is organized as follows. In Section \ref{s2}, we recall  some needed information on the weight vectors of monomials, the admissible monomials in $P_n$ and Singer's criterion on the hit monomials.  In Sections \ref{s3}, we present a construction for the $\mathcal A$-generators of $P_n$ and some properties of it. A basis of $QP_5$ in the degree $d = 15.2^s - 5$ will be explicitly determined in Section \ref{s4}. In Section \ref{s5}, we show that Singer's conjecture for the fifth algebraic transfer is true in the internal degree $d$ by computing the space $(QP_5)_{d}^{GL_5}$. Finally, in Section \ref{s6} we list all admissible monomials of the degree $d$ in $P_5$.

\section{Preliminaries}\label{s2}
\setcounter{equation}{0}

In this section, we recall some needed information from Kameko~\cite{ka} and Singer \cite{si2} on the weight vector of a monomial, the admissible monomials and Kameko's homomorphism.

\subsection{The weight vector of a monomial}\

\begin{notas} We denote $\mathbb N_n = \{1,2, \ldots , n\}$, $ \mathbb J = \{j_1,j_2,\ldots , j_t\}\subset \mathbb N_n$ and
\begin{align*}
X_{\mathbb J} = X_{\{j_1,j_2,\ldots , j_t\}} =
 \prod_{j\in \mathbb N_n\setminus \mathbb J}x_j.
\end{align*}
In particular, $X_{\mathbb N_n} =1,\
X_\emptyset = x_1x_2\ldots x_n,$ 
$X_j = x_1\ldots \hat x_j \ldots x_n, \ 1 \leqslant j \leqslant n,$ and $X:=X_n = x_1x_2\ldots x_{n-1}\in P_{n-1}.$

Let $\alpha_i(a)$ denote the $i$-th coefficient  in dyadic expansion of a non-negative integer $a$. That means
$a= \alpha_0(a)2^0+\alpha_1(a)2^1+\alpha_2(a)2^2+ \ldots ,$ for $ \alpha_i(a) =0$ or 1 with $i\geqslant 0$. 

Let $x=x_1^{a_1}x_2^{a_2}\ldots x_n^{a_n} \in P_n$. Denote $\nu_t(x) = a_t, 1 \leqslant t \leqslant n$.  
Set 
$$\mathbb J_i(x) = \{t \in \mathbb N_n :\alpha_t(\nu_i(x)) =0\},$$
for $i\geqslant 0$. Then, we have
$x = \prod_{i\geqslant 0}X_{\mathbb J_i(x)}^{2^i}.$ 
\end{notas}
\begin{defns}
For a monomial  $x$ in $P_n$,  define two sequences associated with $x$ by
\begin{align*} 
\omega(x)&=(\omega_1(x),\omega_2(x),\ldots , \omega_i(x), \ldots),\\
\sigma(x) &= (\nu_1(x),\nu_2(x),\ldots ,\nu_n(x)),
\end{align*}
where
$\omega_i(x) = \sum_{1\leqslant j \leqslant n} \alpha_{i-1}(\nu_j(x))= \deg X_{\mathbb J_{i-1}(x)},\ i \geqslant 1.$
The sequences $\omega(x)$ and $\sigma(x)$ is respectively called  the weight vector and the exponent vector of $x$. 

Let $\omega=(\omega_1,\omega_2,\ldots , \omega_i, \ldots)$ be a sequence of non-negative integers.  The sequence $\omega$  are called  the weight vector if $\omega_i = 0$ for $i \gg 0$.
\end{defns}

The sets of all the weight vectors and the exponent vectors are given the left lexicographical order. 

  For a  weight vector $\omega$,  we define $\deg \omega = \sum_{i > 0}2^{i-1}\omega_i$.  If  there are $i_0=0, i_1, i_2, \ldots , i_r > 0$ such that $i_1 + i_2 + \ldots + i_r  = m$, $\omega_{i_1+\ldots +i_{s-1} + t} = b_s, 1 \leqslant t \leqslant i_s, 1 \leqslant s \leqslant  r$, and $\omega_i=0$ for all $i > m$, then we write $\omega = (b_1^{(i_1)},b _2^{(i_2)},\ldots , b_r^{(i_r)})$. Denote $b_u^{(1)} = b_u$.  For example, $\omega = (3,2,2,2,1,1,0,\ldots) = (3,2^{(3)},1^{(2)})$.

Denote by   $P_n(\omega)$ the subspace of $P_n$ spanned by all monomials $y$ such that
$\deg y = \deg \omega$, $\omega(y) \leqslant \omega$, and by $P_n^-(\omega)$ the subspace of $P_n$ spanned by all monomials $y \in P_n(\omega)$  such that $\omega(y) < \omega$. Denote by $\mathcal A^+_s$ the subspace of $\mathcal A$ spanned by all $Sq^j$ with $1\leqslant j < 2^s$.

\begin{defns}\label{dfn2} Let $\omega$ be a sequence of non-negative integers and $f, g$ two polynomials  of the same degree in $P_n$. 

i) $f \equiv g$ if and only if $f - g \in \mathcal A^+P_n$. 

ii) $f \equiv_{\omega} g$ if and only if $f - g \in \mathcal A^+P_n+P_n^-(\omega)$. 

iii) $f \simeq_{(s,\omega)} g$ if and only if $f - g \in \mathcal A^+_sP_n+P_n^-(\omega)$.

\end{defns}

Obviously, the relations $\equiv$,  $\equiv_{\omega}$ and $\simeq_{(s,\omega)}$ are equivalence ones. If $f \equiv_{\omega} g$, then $f \simeq_{(s,\omega)} g$ for some $s \geqslant 0$. If $x$ is a monomial in $P_n$ and $\omega = \omega(x)$, then we denote $x \simeq_{s}g$  if and only if $x \simeq_{(s,\omega(x))}g$. 

\begin{props}[see \cite{su5}]\label{mdcb4}  Let $x, y$ be  monomials and let $f, g $ be polynomials in $P_n$ such that $\deg x = \deg f $, $\deg y = \deg g$.

{\rm i)} If $\omega_i(x) \leqslant 1$ for $i > s$ and $x\simeq_{t}f$ with $t \leqslant s$, then $xy^{2^s}\simeq_{t}f y^{2^s}$. 

{\rm ii)} If  $\omega_i(x) = 0$ for $i > s $, $x\simeq_{s} f$ and $y \simeq_r g$, then $xy^{2^s} \simeq_{s+r} fg^{2^s}$.
\end{props}

Denote by $QP_n(\omega)$ the quotient of $P_n(\omega)$ by the equivalence relation $\equiv_\omega$. Then, we have 
$$QP_n(\omega)= P_n(\omega)/ ((\mathcal A^+P_n\cap P_n(\omega))+P_n^-(\omega)).$$  
It is easy to see that
$(QP_n)_d \cong \bigoplus_{\deg \omega = d}QP_n(\omega).$

We note that the weight vector of a monomial is invariant under the permutation of the generators $x_i$, hence $QP_n(\omega)$ has an action of the symmetric group $\Sigma_n$. Furthermore, we have the following.

\begin{lems}[see \cite{su3}]\label{bdm} Let $\omega$ be a weight vector. Then,  $QP_n(\omega)$ is an $GL_n$-module.
\end{lems}

For a  polynomial $f \in  P_n$, we denote by $[f]$ the class in $QP_n$ represented by $f$. If  $\omega$ is a weight vector, then denote by $[f]_\omega$ the class by the equivalence relation $\equiv_\omega$ which is represented by $f$. Denote by $|S|$ the cardinal of a set $S$.

\medskip
For $1 \leqslant i \leqslant n$, define the $\mathcal{A}$-homomorphism  $\rho_i:P_n \to P_n$, which is determined by $\rho_i(x_i) = x_{i+1}, \rho_i(x_{i+1}) = x_i$, $\rho_i(x_j) = x_j$ for $j \ne i, i+1,\ 1 \leqslant i < n$, and $\rho_n(x_1) = x_1+x_2$,  $\rho_n(x_j) = x_j$ for $j > 1$.  

It is easy to see that the general linear group $GL_n$ is generated by the matrices associated with $\rho_i,\ 1\leqslant i \leqslant n,$ and the symmetric group $\Sigma_n$ is generated by the ones associated with $\rho_i,\ 1 \leqslant i < n$. 
So, a class $[f]_\omega$ represented by a homogeneous polynomial $f \in P_n$ is an $GL_n$-invariant if and only if $\rho_i(f) \equiv_\omega f$ for $1 \leqslant i\leqslant n$.  $[f]_\omega$ is an $\Sigma_n$-invariant if and only if $\rho_i(f) \equiv_\omega f$ for $1 \leqslant i < n$. 

\subsection{The admissible monomial}\

\begin{defns}[see Kameko \cite{ka}]\label{defn3} Let $x, y$ be monomials of the same degree in $P_n$. We say that $x <y$ if and only if one of the following holds:  

i) $\omega (x) < \omega(y)$;

ii) $\omega (x) = \omega(y)$ and $\sigma(x) < \sigma(y).$
\end{defns}

\begin{defns}[see Kameko \cite{ka}]
A monomial $x$ is said to be inadmissible if there exist monomials $y_1,y_2,\ldots, y_m$ such that $y_t<x$ for $t=1,2,\ldots , m$ and $x - \sum_{t=1}^my_t \in \mathcal A^+P_n.$ 

A monomial $x$ is said to be admissible if it is not inadmissible.
\end{defns}

Obviously, the set of all the admissible monomials of degree $d$ in $P_n$ is a minimal set of $\mathcal{A}$-generators for $P_n$ in degree $d$. 

\begin{thms}[See Kameko \cite{ka}, Sum \cite{su2}]\label{dlcb1}  
 Let $x, y, w$ be monomials in $P_n$ such that $\omega_i(x) = 0$ for $i > r>0$, $\omega_s(w) \ne 0$ and   $\omega_i(w) = 0$ for $i > s>0$.

{\rm i)}  If  $w$ is inadmissible, then  $xw^{2^r}$ is also inadmissible.

{\rm ii)}  If $w$ is strictly inadmissible, then $wy^{2^{s}}$ is also strictly inadmissible.
\end{thms} 

Now, we recall a result of Singer \cite{si2} on the hit monomials in $P_n$. 

\begin{defns}[see Singer \cite{si2}]\label{spi}  A monomial $z$ in $P_n$   is called a spike if $\nu_j(z)=2^{s_j}-1$ for $s_j$ a non-negative integer and $j=1,2, \ldots , n$. If $z$ is a spike with $s_1>s_2>\ldots >s_{r-1}\geqslant s_r>0$ and $s_j=0$ for $j>r,$ then it is called the minimal spike.
\end{defns}

For a positive integer $d$, by $\mu(d)$ one means the smallest number $r$ for which it is possible to write $d = \sum_{1\leqslant i\leqslant r}(2^{s_i}-1),$ where $s_i >0$. 
 Singer showed in \cite{si2} that if $\mu(d) \leqslant n$, then there exists uniquely a minimal spike of degree $n$ in $P_n$. 

\begin{lems}[See \cite{sp}]\label{bdbs}
 All the spikes in $P_n$ are admissible and their weight vectors are weakly decreasing. 
Furthermore, if a weight vector $\omega$ is weakly decreasing and $\omega_1 \leqslant n$, then there is a spike $z$ in $P_n$ such that $\omega (z) = \omega$. 
\end{lems}

The following is a criterion for the hit monomials in $P_n$.

\begin{thms}[See Singer~\cite{si2}]\label{dlsig} Suppose $x \in P_n$ is a monomial of degree $d$, where $\mu(d) \leqslant n$. Let $z$ be the minimal spike of degree $d$. If $\omega(x) < \omega(z)$, then $x$ is hit.
\end{thms}

From this theorem we see that if $z$ is the minimal spike of degree $d$, $\omega = \omega(z)$ and $f \in (P_n)_d$ then $[f]_\omega = [f]$.  This result also implies a result of Wood, which originally is a conjecture of Peterson~\cite{pe}.
 
\begin{thms}[See Wood~\cite{wo}]\label{dlmd1} 
If $\mu(d) > n$, then $(QP_n)_d = 0$.
\end{thms} 

\subsection{Kameko's homomorphism}\

\medskip
One of the main tools in the study of the hit problem is  Kameko's homomorphism 
$\widetilde{Sq}^0_*: QP_n \to QP_n$. 
This homomorphism is induced by the $\mathbb F_2$-linear map $\psi:P_n\to P_n$, given by
$$
\psi(x) = 
\begin{cases}y, &\text{if }x=x_1x_2\ldots x_ny^2,\\  
0, & \text{otherwise,} \end{cases}
$$
for any monomial $x \in P_n$. Note that $\psi$ is not an $\mathcal A$-homomorphism. However, 
$\psi Sq^{2i} = Sq^{i}\psi,$ and $\psi Sq^{2i+1} = 0$ for any non-negative integer $i$. 

\begin{thms}[Kameko~\cite{ka}]\label{dlmd2} 
Let $m$ be a positive integer. If $\mu(2m+n)=n$, then 
$\widetilde{Sq}^0_*= (\widetilde{Sq}^0_*)_{(n,m)}: (QP_n)_{2m+n}\to (QP_n)_m$ 
is an isomorphism of the $GL_n$-modules.
\end{thms}

We end this section by recalling some notations which will be used in the next sections.
We set 
\begin{align*} 
P_n^0 &=\langle\{x=x_1^{a_1}x_2^{a_2}\ldots x_n^{a_n} \ : \ a_1a_2\ldots a_n=0\}\rangle,
\\ P_n^+ &= \langle\{x=x_1^{a_1}x_2^{a_2}\ldots x_n^{a_n} \ : \ a_1a_2\ldots a_n>0\}\rangle. 
\end{align*}

It is easy to see that $P_n^0$ and $P_n^+$ are the $\mathcal{A}$-submodules of $P_n$. Furthermore, we have the following.

\begin{props}\label{2.7} We have a direct summand decomposition of the $\mathbb F_2$-vector spaces
$QP_n =QP_n^0 \oplus  QP_n^+.$
Here $QP_n^0 = \mathbb F_2\otimes_{\mathcal A}P_n^0$ and  $QP_n^+ = \mathbb F_2\otimes_{\mathcal A}P_n^+$.
\end{props}

\section{A construction for the generators  of $P_n$ }\label{s3}

\medskip
We denote
$$\mathcal N_n =\{(i;I)\ :\ I=(i_1,i_2,\ldots,i_r),1 \leqslant  i < i_1 <  \ldots < i_r\leqslant  n,\ 0\leqslant r <n\}.$$
Let $(i;I) \in \mathcal N_n$ and $j \in \mathbb N_n$. Denote  by $r = \ell(I)$ the length of $I$, and
$$I\cup j = \begin{cases}I,&  \text { if  } j \in I,\\   
(i_1,\ldots ,i_{t-1},j,i_t,\ldots ,i_r), &\text { if  } i_{t-1}<j<i_t, \ 1 \leqslant t \leqslant r+1.
 \end{cases}$$
Here $i_0 = 0$ and $i_{r+1} = n+1$. 
\begin{defn}[see \cite{su5}] Let $(i;I) \in \mathcal N_n$, $r = \ell(I)$, and let $u$ be an integer with $1 \leqslant  u \leqslant r$. A monomial $x \in P_{n-1}$ is said to be $u$-compatible with $(i;I)$ if all of the following hold:

\smallskip
i) $\nu_{i_1-1}(x)= \nu_{i_2-1}(x)= \ldots = \nu_{i_{(u-1)}-1}(x)=2^{r} - 1$,

ii) $\nu_{i_u-1}(x) > 2^{r} - 1$,

iii) $\alpha_{r-t}(\nu_{i_u-1}(x)) = 1,\ \forall t,\ 1 \leqslant t \leqslant  u$,

iv) $\alpha_{r-t}(\nu_{i_t-1} (x)) = 1,\ \forall t,\ u < t \leqslant r$.
\end{defn}

Clearly, a monomial $x$ can be $u$-compatible with a given $(i;I) \in \mathcal N_n $ for at most one value of $u$. By convention, $x$ is $1$-compatible with $(i;\emptyset)$.

For $1 \leqslant i \leqslant n$, define the homomorphism $f_i: P_{n-1} \to P_n$ of algebras by substituting
$$f_i(x_j) = \begin{cases} x_j, &\text{ if } 1 \leqslant j <i,\\
x_{j+1}, &\text{ if } i \leqslant j <n.
\end{cases}$$

\begin{defn}[see \cite{su5}]\label{dfn1} Let $(i;I) \in \mathcal N_n$, $x_{(I,u)} = x_{i_u}^{2^{r-1}+\ldots + 2^{r-u}}\prod_{u< t \leqslant r}x_{i_t}^{2^{r-t}}$ for $r = \ell(I)>0$, $x_{(\emptyset,1)} = 1$.  For a monomial $x$ in $P_{n-1}$, 
we define the monomial  $\phi_{(i;I)}(x)$ in $P_n$ by setting
$$ \phi_{(i;I)}(x) = \begin{cases} (x_i^{2^r-1}f_i(x))/x_{(I,u)}, &\text{if there exists $u$ such that}\\ &\text{$x$ is $u$-compatible with $(i, I)$,}\\
0, &\text{otherwise.}
\end{cases}$$

Then we have an $\mathbb F_2$-linear map $\phi_{(i;I)}:P_{n-1}\to P_n$.
In particular, $\phi_{(i;\emptyset)} = f_i$.
\end{defn}

It is easy to see that if $\phi_{(i;I)}(x) \ne 0$ then $\omega(\phi_{(i;I)}(x)) = \omega(x)$.

\begin{defn}
For any $(i;I) \in \mathcal N_n$, we define the homomorphism $p_{(i;I)}: P_n \to P_{n-1}$ of algebras by substituting
$$p_{(i;I)}(x_j) =\begin{cases} x_j, &\text{ if } 1 \leqslant j < i,\\
\sum_{s\in I}x_{s-1}, &\text{ if }  j = i,\\  
x_{j-1},&\text{ if } i< j \leqslant n.
\end{cases}$$
Then, $p_{(i;I)}$ is a homomorphism of $\mathcal A$-modules.  In particular, for $I =\emptyset$,  $p_{(i;\emptyset)}(x_i)= 0$  and $p_{(i;I)}(f_i(y)) = y$ for any $y \in P_{n-1}$. 
\end{defn}

For a subset $B \subset P_n$, we denote $[B] = \{[f] : f \in B\}$. If $B \subset P_n(\omega)$, then we set $[B]_\omega = \{[f]_\omega : f \in B\}$. 

We denote
\begin{align*}\Phi^0(B) &= \bigcup_{1\leqslant i \leqslant n}\phi_{(i;\emptyset)}(B) = \bigcup_{1\leqslant i \leqslant n}f_i(B).\\
\Phi^+(B) &= \bigcup_{(i;I)\in \mathcal N_n, 0<\ell(I)\leqslant n-1}\phi_{(i;I)}(B)\setminus P_n^0.\\
\Phi(B) &= \Phi^0(B)\bigcup \Phi^+(B). 
\end{align*} 

 It is easy to see that if $B$ is a minimal set of generators for $\mathcal A$-module $P_{n-1}$ in degree $n$, then  $\Phi^0(B)$  is a minimal set of generators for $\mathcal A$-module $P_n^0$ in degree $n$ and $\Phi^+(B) \subset P_n^+$.

Denote by $B_n(d)$ the set of all the admissible monomials of degree $d$ in $P_n$. Set $B_n^0(d) = B_n(d)\cap P_{n}^0$,  $B_n^+(d) = B_n(d)\cap P_{n}^+$. If $\omega$ is a weight vector of degree $d$, then we set $B_n(\omega) = B_n(d)\cap P_{n}(\omega)$, $B_n^0(\omega) = B_n(\omega)\cap P_{n}^0$, $B_n^+(\omega) = B_n(\omega)\cap P_{n}^+$. 

The following gives our prediction on the relation between the admissible monomials for  the polynomial algebras.

\begin{con}\label{gtom} If $\omega$ is a weight vector, then  
$\Phi(B_{n-1}(\omega)) \subset B_{n}(\omega).$
\end{con}

It is easy to see that if this conjecture is true, then $\Phi(B_{n-1}(d)) \subset B_{n}(d)$ for any positive integer $d$.

From the results of Peterson \cite{pe}, Kameko \cite{ka} and the present author \cite{su5},  we can see that this conjecture is true for $n \leqslant 4$. In Section \ref{s5}, we will show that this conjecture is true for $n = 5$ and any weight vector $\omega$ of the degree $d = 15.2^s -5$ for arbitrary non-negative integer $s$. 
 
We recall some needed results in \cite{su5} which will be used later.  For a positive integer $b$, denote 
$$\omega_{(n,b)} =((n-1)^{(b)}) , \  \tilde \omega_{(n,b)}= ((n-1)^{(b)},1).$$

\begin{lem}[See \cite{su5}]\label{hq0} Let $b$ be a positive integer and let $j_0, j_1, \ldots , j_{b-1} \in \mathbb N_n$. We set $i = \min\{j_0,\ldots , j_{b-1}\}$, $I = (i_1, \ldots, i_r)$ with $\{i_1, \ldots, i_r\} = \{j_0,\ldots , j_{b-1}\}\setminus \{i\}$. Then, we have 
$$\prod_{0 \leqslant t <b}X_{j_t}^{2^t} \simeq_{b-1} \phi_{(i;I)}(X^{2^b-1}).$$
\end{lem}

\begin{lem}[See \cite{su5}]\label{bdad} Let $i, j, d, a, b$ be positive integers such that $i,j\leqslant n$, $i \ne j$, and $a+b = 2^d-1$. Then
 $$ X_i^aX_j^b \simeq_{2} X_i^{2^d-2}X_j.$$
\end{lem}

\begin{lem}[See \cite{su5}]\label{bdbss2} Let $(i;I) \in \mathcal N_n$ and let $d, h, u$ be   integers such that $\ell(I)= r < h \leqslant d$, and $1 \leqslant i< u \leqslant n$. Then, we have
$$\phi_{(i;I)}(X^{2^{h}-1})X_u^{2^{d }-2^{h}} \simeq_{r+2}\phi_{(i;I\cup u)}(X^{2^d-1}).$$
\end{lem}

 We now prove some relations which are the usual tools for studying the hit problem.
\begin{prop}\label{bdbss} For any  integer $0 < \ell \leqslant n$, 
$$X_\ell^{2^\ell-1}x_\ell^{2^\ell} \simeq_\ell \sum_{u=\ell}^n\sum_{(i;I) \in \mathcal N_{\ell - 1}}\phi_{(i;I\cup u)}(X^{2^\ell-1})x_u^{2^\ell} + \sum_{u=\ell+1}^nX_u^{2^\ell-1}x_u^{2^\ell}.$$
\end{prop}

\begin{proof} We prove the proposition by induction on $\ell$. For $\ell = 1$,
$$X_1x_1^2 = \sum_{u=2}^nX_ux_u^{2} + Sq^1(X_\emptyset) \simeq_1 \sum_{u=2}^nX_ux_u^{2}.$$
Since $\mathcal N_0=\emptyset$, the proposition is true. 

Suppose that $1 \leqslant \ell <n$ and the proposition is true for $\ell$. Using the Cartan formula, we have
\begin{align*}
X_{\ell+1}^{2^{\ell + 1} -1}x_{\ell+1}^{2^{\ell+1}} 
&= \sum_{u=1}^\ell X_u^{2^{\ell+1}-1}x_u^{2^{\ell+1}} + \sum_{u=\ell+2}^nX_u^{2^{\ell+1} - 1}x_u^{2^{\ell+1} } + Sq^1(X_\emptyset^{2^{\ell+1}-1})\\
&\simeq_1 \sum_{u=1}^\ell X_u^{2^{\ell +1-u}-1}(X_u^{2^{u}-1}x_u^{2^{u}})^{2^{\ell+1-u}} + \sum_{u=\ell+2}^nX_u^{2^{\ell+1} - 1}x_u^{2^{\ell+1} }.
\end{align*}
For $1\leqslant u \leqslant \ell$, using the inductive hypothesis and  Proposition \ref{mdcb4}, we get 
\begin{align*}
X_u^{2^{\ell +1-u}-1}(X_u^{2^{u}-1}x_u^{2^{u}})^{2^{\ell+1-u}} 
&\simeq_{\ell+1}X_u^{2^{\ell +1-u}-1}\Big(\sum_{t=u+1}^nX_t^{2^u-1}x_t^{2^u}\\ 
&\quad +\sum_{t=u}^n\sum_{(i;I) \in \mathcal N_{u-1}}\phi_{(i;I\cup t)} (X^{2^{u}-1})x_t^{2^{u}}\Big)^{2^{\ell+1-u}}.
\end{align*}
According to Lemma \ref{hq0} and Proposition \ref{mdcb4}, 
\begin{align*}
X_u^{2^{\ell +1-u}-1}(X_t^{2^u-1}x_t^{2^u})^{2^{\ell+1-u}} &\simeq_{\ell }\phi_{(u;t)}(X^{2^{\ell+1} - 1})x_t^{2^{\ell+1}},\\
X_u^{2^{\ell +1-u}-1}(\phi_{(i;I\cup t)} (X^{2^{u}-1})x_t^{2^{u}})^{2^{\ell+1-u}} &\simeq_{\ell} \phi_{(i;I\cup t \cup u)} (X^{2^{\ell+1}-1})x_t^{2^{\ell+1}}. 
\end{align*}
From the above equalities, we get
\begin{align*}
&X_u^{2^{\ell +1-u}-1}(X_u^{2^{u}-1}x_u^{2^{u}})^{2^{\ell+1-u}}
\simeq_{\ell+1}  \sum_{(i;I) \in \mathcal N_{u-1}}\phi_{(i;I\cup u)} (X^{2^{\ell+1}-1})x_u^{2^{\ell+1}}\\
&\qquad+  \sum_{t=u+1}^n\Big(\sum_{(i;I) \in \mathcal N_{u-1}}\phi_{(i;I\cup t \cup u)} (X^{2^{\ell+1}-1})x_t^{2^{\ell+1}} + \phi_{(u;t)}(X^{2^{\ell+1} - 1})x_t^{2^{\ell+1}}\Big).
\end{align*}
For $2 \leqslant  u \leqslant h \leqslant n$, we set $\mathcal N_{u-1}\cup u = \{(i;I\cup u) : (i;I) \in  \mathcal N_{u-1}\}$. Then we have
\begin{equation}\label{rela1}\mathcal N_{h} = (\mathcal N_1\cup 2) \cup \ldots \cup (\mathcal N_{h-1}\cup h) \cup \{(1;\emptyset),\ldots ,(h;\emptyset)\}.
\end{equation}
By a direct computation from the above equalities and using the relation (\ref{rela1}), we have
\begin{align*}
&\sum_{u=1}^\ell X_u^{2^{\ell +1-u}-1}(X_u^{2^{u}-1}x_u^{2^{u}})^{2^{\ell+1-u}}\simeq_{\ell+1}\sum_{u=2}^\ell\sum_{(i;I) \in \mathcal N_{u-1}}\phi_{(i;I\cup u)} (X^{2^{\ell+1}-1})x_u^{2^{\ell+1}}\\
&\quad + \sum_{t=2}^\ell\Big(\sum_{u=1}^{t-1}\sum_{(i;I) \in \mathcal N_{u-1}\cup u}\phi_{(i;I\cup t)} (X^{2^{\ell+1}-1}) + \phi_{(u;t)}(X^{2^{\ell+1} - 1})\Big)x_t^{2^{\ell+1}}\\
&\quad +\sum_{t=\ell+1}^n\Big(\sum_{u=1}^{\ell}\sum_{(i;I) \in \mathcal N_{u-1}\cup u}\phi_{(i;I\cup t)} (X^{2^{\ell+1}-1}) + \phi_{(u;t)}(X^{2^{\ell+1} - 1})\Big)x_t^{2^{\ell+1}}\\
&=\sum_{u=2}^\ell\sum_{(i;I) \in \mathcal N_{u-1}}\phi_{(i;I\cup u)} (X^{2^{\ell+1}-1})x_u^{2^{\ell+1}}
+ \sum_{t=2}^\ell\sum_{(i;I) \in \mathcal N_{t-1}}\phi_{(i;I\cup t)} (X^{2^{\ell+1}-1})x_t^{2^{\ell+1}}\\
&\hskip5.5cm +\sum_{t=\ell+1}^n\sum_{(i;I) \in \mathcal N_{\ell}}\phi_{(i;I\cup t)} (X^{2^{\ell+1}-1})x_t^{2^{\ell+1}}\\ 
&=\sum_{t=\ell+1}^n\sum_{(i;I) \in \mathcal N_{\ell}}\phi_{(i;I\cup t)} (X^{2^{\ell+1}-1})x_t^{2^{\ell+1}}.
\end{align*}
Combining the above equalities, we get
$$X_{\ell+1}^{2^{\ell + 1} -1}x_{\ell+1}^{2^{\ell+1}}  \simeq_{\ell+1}  \sum_{u=\ell+1}^n \sum_{(i;I) \in \mathcal N_{\ell}}\phi_{(i;I\cup u )} (X^{2^{\ell+1}-1})x_u^{2^{\ell+1}}+ \sum_{u=\ell+2}^nX_u^{2^{\ell+1} - 1}x_u^{2^{\ell+1}}.$$
The proposition follows.
\end{proof}

From the proof of this lemma, we obtain the following.
\begin{corl}\label{hq3} For $2 \leqslant d \leqslant n$, we have
$$ \sum_{u=1}^{d-1}X_u^{2^{d}-1}x_u^{2^{d}} \simeq_{d} \sum_{u=d}^n \sum_{(i;I) \in \mathcal N_{d-1}}\phi_{(i;I\cup u )} (X^{2^{d}-1})x_u^{2^{d}}.$$
\end{corl}
  
\begin{prop}\label{bdbss0} Let $d, h,t, u$ be  integers such that $1 \leqslant u < d-n+t$, $0< t <h \leqslant  n$, and let $x$ be a monomial in $P_n$. Then we have
$$Z:=\phi_{(t;I_t)}(X^{2^{d - u}-1})X_h^{2^{d }-2^{d - u}}x^{2^d} \simeq_{n-t+1}\phi_{(t;I_t)}(X^{2^d-1})x^{2^d}.$$
\end{prop}
\begin{proof} We prove the proposition by double induction on $(t, u)$. If $t = n-1$, then $h=n$. By Lemma \ref{bdad}, we have
\begin{align*}\phi_{(n-1;n)}(X^{2^{d - u}-1})X_n^{2^{d }-2^{d - u}}x^{2^d} &=X_{n-1}^{2^{d - u}-2}X_n^{2^{d }-2^{d - u}+1}x^{2^d}\\ & \simeq_2  \phi_{(n-1;n)}(X^{2^{d}-1})x^{2^d}.
\end{align*}
So, the proposition holds.
Suppose that $0< t < n-1$ and the proposition is true for $t+1$. If $h = t+1$, then
$$Z = \phi_{(t+2;I_{t+2})}(X^{2^{n-t-1}-1})
(X_t^{2^{d -n+t - u+1}-2}X_{t+1}^{2^{d -n+t+1}-2^{d-n+t- u+1} + 1}x^{2^{d-n+t+1}})^{2^{n-t-1}}.$$
According to Lemma \ref{bdad}, 
$$X_t^{2^{d -n+t - u+1}-2}X_{t+1}^{2^{d -n+t+1}-2^{d-n+t- u+1} + 1}x^{2^{d-n+t+1}}\simeq_2 X_t^{2^{d -n+t +1}-2}X_{t+1}x^{2^{d-n+t+1}}.$$
Hence,  using Proposition \ref{mdcb4}, we obtain
\begin{align*}
Z &\simeq_{n-t+1} \phi_{(t+2;I_{t+2})}(X^{2^{n-t-1}-1})(X_t^{2^{d -n+t +1}-2}X_{t+1}x^{2^{d-n+t+1}})^{2^{n-t-1}}\\
&=\phi_{(t;I_{t})}(X^{2^{d}-1})x^{2^{d}}. 
\end{align*}
The proposition holds. Suppose that $h>t+1$ and $u =1$. We have
$$Z = \phi_{(t+1;I_{t+1})}(X^{2^{n-t}-1})
(X_t^{2^{d -n+t - 1}-1}X_{h}^{2^{d -n+t-1}}x^{2^{d-n+t}})^{2^{n-t}}.$$
Since $X_t^{2^{d - n + t - 1}-1}X_{h}^{2^{d -n+t-1}}x^{2^{d-n+t}} \simeq_{1} X_t^{2^{d -n+t - 1}}X_{h}^{2^{d -n+t-1}-1}x^{2^{d-n+t}}$, applying Proposition \ref{mdcb4} and the inductive hypothesis, we have
\begin{align*}Z &\simeq_{n-t+1} \phi_{(t+1;I_{t+1})}(X^{2^{n-t}-1})
(X_t^{2^{d - n + t - 1}}X_{h}^{2^{d -n+t-1}-1}x^{2^{d-n+t}})^{2^{n-t}}\\
&= \phi_{(t+1;I_{t+1})}(X^{2^{n-t}-1})
X_{h}^{2^{d -1}- 2^{n-t}}(X_tx^2)^{2^{d - 1}}\\
&\simeq_{n-t} \phi_{(t+1;I_{t+1})}(X^{2^{d-1}-1})(X_tx^2)^{2^{d - 1}}\\
&= \phi_{(t+2;I_{t+2})}(X^{2^{n-t-1}-1})(X_{t+1}^{2^{d-n+t}-1}X_t^{2^{d -n+t}}x^{2^{d-n+t+1}})^{2^{n-t-1}}. \end{align*}
According to Lemma \ref{bdad}, 
$$X_{t+1}^{2^{d-n+t}-1}X_t^{2^{d -n+t}}x^{2^{d-n+t+1}}\simeq_2 X_t^{2^{d -n+t+1}-2}X_{t+1}x^{2^{d-n+t+1}}.$$
Hence,  using Proposition \ref{mdcb4}, one gets
\begin{align*}Z &\simeq_{n-t+1} 
 \phi_{(t+2;I_{t+2})}(X^{2^{n-t-1}-1})(X_{t}^{2^{d-n+t+1}-2}X_{t+1}x^{2^{d-n+t+1}})^{2^{n-t-1}}\\
&= \phi_{(t;I_{t})}(X^{2^{d}-1})x^{2^{d}}. \end{align*}

Now, suppose that $h> t+1$ and $u > 1$. Set $d' = d-u+1$. Combining Proposition \ref{mdcb4} and the inductive hypothesis gives
\begin{align*}Z &=  \phi_{(t;I_{t})}(X^{2^{d'-1}-1})X_{h}^{2^{d'}-2^{d'-1}}(X_{h}^{2^{d-d'} - 1}x^{2^{d-d'}})^{2^{d'}}\\ & \simeq_{n-t+1}\phi_{(t;I_{t})}(X^{2^{d'}-1})(X_{h}^{2^{d-d'} - 1}x^{2^{d-d'}})^{2^{d'}}\\ &= \phi_{(t;I_{t})}(X^{2^{d-u+1}-1})X_{h}^{2^{d} - 2^{d-u+1}}x^{2^d}
\simeq_{n-t+1} \phi_{(t;I_{t})}(X^{2^{d}-1})x^{2^d}. \end{align*}
The proposition is proved.
\end{proof}

\begin{prop}\label{bdbss1} For any  integer $d \geqslant n\geqslant 2$, 
$$X_n^{2^d-1}x_n^{2^d} \simeq_{n}\sum_{(i;I) \in \mathcal N_{n - 1}}\phi_{(i;I\cup n)}(X^{2^d-1})x_n^{2^d}.$$
\end{prop}

\begin{proof} By Proposition \ref{bdbss}, we have
$$X_n^{2^n-1}x_n^{2^n} \simeq_{n} \sum_{(i;I) \in \mathcal N_{n - 1}}\phi_{(i;I\cup n)}(X^{2^n-1})x_n^{2^n}.$$
Hence,  using Proposition \ref{mdcb4}(i), we get
$$X_n^{2^d-1}x_n^{2^d} = X_n^{2^n-1}x_n^{2^n}X_\emptyset^{2^d - 2^n} \simeq_{n} \sum_{(i;I) \in \mathcal N_{n - 1}}\phi_{(i;I\cup n)}(X^{2^n-1})X_n^{2^d - 2^n}x_n^{2^d}.$$

Let $(i;I) \in \mathcal N_{n - 1}$. If $I = \emptyset$, then using Lemma \ref{bdad}, we have 
\begin{align*}&\phi_{(i;I\cup n)} (X^{2^{n}-1})X_n^{2^{d}-2^n}x_n^{2^d} = \phi_{(i;n)} (X^{2^{n}-1})X_n^{2^{d}-2^n}x_n^{2^d}\\
&= X_i^{2^n-2}X_n^{2^{d}-2^n+1}x_n^{2^d}
 \simeq_2 X_i^{2^d-2}X_nx_n^{2^d}
=  \phi_{(i;I\cup n)} (X^{2^{d}-1})x_n^{2^d}.
\end{align*}
If $I = (i_1, \ldots , i_r) , r > 0$, then $s = n - \ell (I\cup n) > 0$. Hence,
\begin{align*}Y :&= \phi_{(i;I\cup n)} (X^{2^{n}-1})X_n^{2^{d}-2^n}x_n^{2^d} \\ &= 
\phi_{(i_1;I\cup n\setminus i_1)} (X^{2^{n-s}-1})(X_i^{2^s-1}X_n^{2^{d -n+s}-2^s}x_n^{2^{d-n+s}})^{2^{n-s}}.
\end{align*}
 By Lemma \ref{bdad}, 
$$X_i^{2^s-1}X_n^{2^{d - n +s}-2^s}x_n^{2^{d-n+s}} \simeq_2 X_i^{2^{d - n +s}-2}X_nx_n^{2^{d-n+s}}.$$

If $(i;I\cup n) = (1;I_1)$, then using Proposition \ref{bdbss0}, we have
\begin{align*}
\phi_{(1;I_1)} (X^{2^{n}-1})X_n^{2^{d}-2^n}x_n^{2^d} \simeq_{n}  \phi_{(1;I_1)} (X^{2^{d}-1})x_n^{2^d}.
\end{align*}

If $(i;I\cup n) \ne (1;I_1)$, then $s \geqslant 2$. Using Proposition \ref{mdcb4} and Lemma \ref{hq0}, we obtain
\begin{align*}
Y&\simeq_{n-s+2} \phi_{(i_1;I\cup n\setminus i_1)} (X^{2^{n-s}-1})(X_nX_i^{2^{d -n+s}-2}x_n^{2^{d-n+s}})^{2^{n-s}}\\
&= \Big(\phi_{(i_1;I\cup n\setminus i_1)}(X^{2^{n-s}-1})X_n^{2^{n-s}}X_i^{2^{n-s+1}}\Big)X_i^{2^{d}-2^{n-s+2}}x_n^{2^d}\\
&\simeq_{n-s+1} \phi_{(i;I\cup n)}(X^{2^{n-s+2}-1})X_i^{2^{d}-2^{n-s+2}}x_n^{2^d}
 = \phi_{(i;I\cup n)}(X^{2^{d}-1})x_n^{2^d}.
\end{align*}
Since $n-s+2 \leqslant n$, the proposition follows.
\end{proof}

Now, using the above results we prove the following which has been proved in \cite{su5} by another method.
Denote by $I_t =(t+1, t+2, \ldots, n)$ for $1 \leqslant t  \leqslant n$. Set  
$$Y_t = Y_{n,t} = \sum_{u=t}^n\phi_{(t;I_t)}(X^{2^{d}-1})x_u^{2^{d}},\  d > n - t + 1.$$ 

\begin{prop}[See \cite{su5}]\label{hq4} For $1 < t \leqslant n$, 
$$Y_t \simeq_{(n - t + 1,\omega)} \sum_{(j;J)}\phi_{(j;J)}(X^{2^d-1})x_j^{2^d},$$ 
where the sum runs over some $(j;J)\in \mathcal N_n$ with $1 \leqslant j < t$, $J \subset I_{t-1}$, $J \ne I_{t-1}$ and $\omega=\omega(X_1^{2^d-1}x_1^{2^d})$. 
\end{prop}

\begin{lem}\label{bd5} For $d\geqslant n$,
$Y_1 = \sum_{u=1}^n\phi_{(1;I_1)}(X^{2^d-1})x_u^{2^d} \simeq_{(n,\omega)} 0$, with  $\omega = {\omega}(X_1^{2^d-1}x_1^{2^d})$. More precisely, 
$$Y_1 = \sum_{0\leqslant i < n} Sq^{2^i}(y_i) + h,$$
with $y_i$ polynomials in $P_n$,  and $h \in P_n^-({\omega})$.
\end{lem}
\begin{proof} First, we prove the following relation by induction on $t \geqslant 1.$
\begin{equation}\label{rela3}
Y_1 \simeq_{(n,\omega)} Y_t + \sum_{u=t}^n \sum_{(i;I) \in \mathcal N_{t-1}}\phi_{(i;I\cup I_{t-1} )} (X^{2^{d}-1})x_u^{2^{d}}.
\end{equation}

Since $\mathcal N_0 = \emptyset$, the relation (\ref{rela3}) is true for $t=1$.  Note that if $1\leqslant t <n$, then
$$\phi_{(t;I_t)}(X^{2^{d}-1})x_t^{2^{d}} = \phi_{(t+1;I_{t+1})}(X^{2^{n-t}-1})(X_t^{2^{t}-1}x_t^{2^{t}}X_\emptyset^{2^{d-n+t}-2^t})^{2^{n-t}}.$$

Applying Proposition \ref{bdbss} and Proposition \ref{mdcb4}, we have
\begin{align*}X_t^{2^{t}-1}x_t^{2^{t}}&X_\emptyset^{2^{d-n+t}-2^t} \simeq_{t}  \sum_{u=t+1}^nX_u^{2^{d-n+t}-1}x_u^{2^{d-n+t}}\\ 
&+ \sum_{u=t}^n\sum_{(i;I) \in \mathcal N_{t - 1}}\phi_{(i;I\cup u)}(X^{2^t-1})X_u^{2^{d - n + t}-2^t}x_u^{2^{d - n + t}} .
\end{align*}
Using Proposition \ref{bdbss0}, we obtain
$$\phi_{(t+1;I_{t+1})}(X^{2^{n-t}-1})X_u^{2^{d}-2^{n-t}}x_u^{2^{d}} \simeq_{n-t} \phi_{(t+1;I_{t+1})}(X^{2^{d}-1})x_u^{2^{d}}.$$
Applying Lemma \ref{hq0} and Proposition \ref{mdcb4}, we have
\begin{align*}\phi_{(t+1;I_{t+1})}(X^{2^{n-t}-1})&(\phi_{(i;I\cup u)}(X^{2^t-1}))^{2^{n-t}}X_u^{2^{d}-2^n}x_u^{2^{d}}\\
&\simeq_{n-1} \phi_{(i;I\cup I_t\cup u)}(X^{2^n-1}))X_u^{2^{d}-2^n}x_u^{2^{d}}
\end{align*}
Combining the above equalities and Proposition \ref{mdcb4} gives
$$\phi_{(t;I_t)}(X^{2^{d}-1})x_t^{2^{d}}\simeq_{n}  Y_{t+1} 
+ \sum_{u=t}^n\sum_{(i;I) \in \mathcal N_{t - 1}}\phi_{(i;I\cup I_t\cup u)}(X^{2^n-1})X_u^{2^{d}-2^n}x_u^{2^{d}} .
$$
If  $(i;I\cup I_t\cup u) = (1;I_1)$, then by Proposition \ref{bdbss0}, we have
$$
\phi_{(1;I_1)}(X^{2^n-1})X_u^{2^{d}-2^n}x_u^{2^{d}} \simeq_n \phi_{(1;I_1)}(X^{2^{d} -1})x_u^{2^{d}}.
$$
Suppose $(i;I\cup I_t\cup u) \ne (1;I_1)$. Then, $0< s = \ell(I\cup I_t\cup u)\leqslant n-2$. Using Lemma \ref{bdbss2}, we have
\begin{align*}\phi_{(1;I_1)}(X^{2^n-1})X_u^{2^{d}-2^n}x_u^{2^{d}} \simeq_{s+2} \phi_{(i;I\cup I_t\cup u)}(X^{2^d-1})x_u^{2^{d}}. 
\end{align*}

Combining the above equalities gives
\begin{align*}\phi_{(t;I_t)}(X_t^{2^{d}-1})x_t^{2^{d}}
&\simeq_{n} Y_{t+1} + \sum_{u=t}^n \sum_{(i;I) \in \mathcal N_{t-1}}\phi_{(i;I\cup I_{t}\cup u)} (X^{2^{d}-1})x_u^{2^{d}}.
\end{align*}

Note that  
$I_t \cup t= I_{t-1}$, $I_t \cup u =I_t$ for $u > t$, and 
$\mathcal N_t = \mathcal N_{t-1}\cup   (\mathcal N_{t-1}\cup t) \cup\{(t;\emptyset)\}$. 

Computing from the last equalities and using the inductive hypothesis, we obtain 
\begin{align*}
Y_1 &\simeq_{(n,\omega)} Y_{t+1}  + \sum_{u=t}^n \sum_{(i;I) \in \mathcal N_{t-1}}\phi_{(i;I\cup I_{t}\cup u)} (X^{2^{d}-1})x_u^{2^{d}} \\
&\quad + \sum_{u=t+1}^n\phi_{(t;I_t)}(X^{2^{d}-1})x_u^{2^{d}}
+  \sum_{u=t}^n \sum_{(i;I) \in \mathcal N_{t-1}}\phi_{(i;I\cup I_{t-1} )} (X^{2^{d}-1})x_u^{2^{d}} \\
&= Y_{t+1} + \sum_{u=t+1}^n \sum_{(i;I) \in \mathcal N_{t-1}}\phi_{(i;I\cup I_{t} )} (X^{2^{d}-1})x_u^{2^{d}}\\
&\quad 
+ \sum_{u=t+1}^n \sum_{(i;I) \in \mathcal N_{t-1}\cup t}\phi_{(i;I\cup I_{t})} (X^{2^{d}-1})x_u^{2^{d}} + \sum_{u=t+1}^n\phi_{(t;I_t)}(X^{2^{d}-1})x_u^{2^{d}}\\
&= Y_{t+1} + \sum_{u=t+1}^n \sum_{(i;I) \in \mathcal N_{t}}\phi_{(i;I\cup I_{t} )} (X^{2^{d}-1})x_u^{2^{d}}.
\end{align*}
The relation (\ref{rela3}) is proved.

Since $Y_n = X_n^{2^d-1}x_n^{2^d}$, using  the relation (\ref{rela3}) with $t=n$ and  Proposition \ref{bdbss}, one gets
$$Y_1 \simeq_{(n,\omega)} X_n^{2^d-1}x_n^{2^d} + \sum_{(i;I) \in \mathcal N_{n-1}}\phi_{(i;I\cup n)} (X^{2^{d}-1})x_n^{2^{d}} \simeq_{(n,\omega)} 0.$$
The lemma is proved.
\end{proof}

\begin{proof}[Proof of Proposition \ref{hq4}] We have $Y_t = Z^{2^d-1}Y_1$ with $Z = x_1x_2\ldots x_{t-1}$ and $Y_1 = Y_1(x_t, \ldots ,x_n) \in P_{n-t+1}:= \mathbb F_2[x_t, \ldots,x_n]$. Since $d\geqslant n-t+1$, by Lemma \ref{bd5}, $Y_t$ is a sum of polynomials of the form $f= Z^{2^d-1}(Sq^{2^i}(y) + h)$ with $ 0\leqslant i \leqslant  n-t$, $y$ a monomial in $P_{n-t+1}$ and  
$h \in P_{n-t+1}^-(\omega^*)\ \text{ with } \ \omega^* = \omega((x_{t+1} \ldots x_n)^{2^d-1}x_t^{2^d}).$

 Then $Z^{2^d-1}h \in P_{n}^-(\omega)$ with $\omega = \omega(X_1^{2^d-1}x_1^{2^d})$. So, using the Cartan formula, we have
\begin{align*} f &\simeq_{(0,\omega)} Sq^{2^i}(Z^{2^d-1}y) + \sum _{1 \leqslant v \leqslant 2^i}Sq^v(Z^{2^d-1})Sq^{2^i - v}(y).
\end{align*} 
By a direct computation using the Cartan formula, we can easily show that if $1\leqslant v < 2^i$, and a monomial $z$ appears as a term of the polynomial $Sq^v(Z^{2^d-1})Sq^{2^i - v}(y)$, then $\omega_u(z) < n-1$ for some $u\leqslant d$. Hence, using the Cartan formula, one gets 
$$f \simeq_{(i+1,\omega)}  Sq^{2^i}(Z^{2^d-1})y \simeq_{(0,\omega)}  \sum _{0< j < t}Z^{2^d-1}x_j^{2^i}y.$$

Since $\omega_u(Z^{2^d-1}x_j^{2^i}) = t-1$ for $ u \leqslant i$  and $\omega_u(Z^{2^d-1}x_j^{2^i}) = t-2$ for $i < u \leqslant d$, if $Z^{2^d-1}x_j^{2^i}y \notin P_{n}^-(\omega)$, then $\omega_u(y) = n-t$ for $u \leqslant i$ and $\omega_u(y) = n-t+1$ for $i < u \leqslant d$. Hence, we obtain
$$ Z^{2^d-1}x_j^{2^i}y = \Big(\prod _{0\leqslant u \leqslant i-1}X_{j_u}^{2^u}\Big)X_j^{2^i}X_j^{2^d-2^{i+1}}x_j^{2^d},$$
where $j_u \geqslant t$, for $0\leqslant u \leqslant i-1$. By convention, $\prod _{0\leqslant u \leqslant i-1}X_{j_u}^{2^u} = 1$ for $i=0$.

 According to Lemma \ref{hq0}, there is $J \subset \{j_0,\ldots , j_{i-1}\} \subset I_{t-1}$  such that 
$$ \Big(\prod _{0\leqslant u \leqslant i-1}X_{j_u}^{2^u}\Big)X_j^{2^i} \simeq_i \phi_{(j;J)}(X^{2^{i+1}-1}).$$
From the above equalities and Proposition \ref{mdcb4}, we get
$$Z^{2^d-1}x_j^{2^i}y \simeq_{(i+1,\omega)} \phi _{(j;J)}(X^{2^{i+1}-1})X_j^{2^d-2^{i+1}}x_j^{2^d}= \phi _{(j;J)}(X^{2^{d}-1})x_j^{2^d}.$$

We have $\ell(J) \leqslant  i \leqslant  n-t <  n-t+1 = \ell(I_{t-1}) $. Hence, $J \ne I_{t-1}$. Since $i+1 \leqslant n-t+1$, the proposition follows.
\end{proof}

\section{The admissible monomials  of the degree $d = 15.2^s - 5$ in $P_5$}\label{s4}

By using Theorem \ref{dlmd1} we can show that the hit problem is reduced to the case of the degree of the form
$d = t(2^s - 1) + 2^sm$
with $t, s, m$ positive integers such that $\mu(m) < t \leqslant n$  (see \cite{su5}). For $t = n$, the problem has been studied by H\uhorn ng \cite{hu} and by T\'in and Sum \cite{ts}. For $t= n = 5$, it is explicitly computed by T\'in \cite{tin0,ti1,tin} with $m = 1,2,3$ and by the present author \cite{su3} with $m = 5$. These results show that Conjecture \ref{gtom} is true in those cases.

In this section we study this problem for $n= t = 5$ and $m = 10$. More precisely, we explicitly determine the admissible monomials of degree $d=15.2^s-5$ in $P_5$. 

It is easy to see that for $s > 1$, we have $\mu(15.2^s-5) = 5$. Hence, Theorem \ref{dlmd2} implies that
$$(\widetilde{Sq}^0_*)^{s-1}: (QP_5)_{15.2^s-5} \longrightarrow (QP_5)_{25}$$
is an isomorphism of $GL_5$-modules for every $s \geqslant 1$. So, we need only to compute the space $(QP_5)_{15.2^s-5}$ for $s=1$. Since Kameko's homomorphism
$$(\widetilde{Sq}^0_*)_{(5,10)}: (QP_5)_{25} \longrightarrow (QP_5)_{10}$$
is an epimorphism, we have $(QP_5)_{25} \cong \text{Ker}(\widetilde{Sq}^0_*)_{(5,10)} \bigoplus (QP_5)_{10}$.

The admissible monomials of the degree 10 in $P_5$ have been determined by T\'in \cite{tin}. We have $B_5(10) = \{a_{10,t} : 1 \leqslant t \leqslant 280\}$, where the monomials $a_{10,t}$, $1 \leqslant t \leqslant 280$,  are listed in Subsection \ref{s61}.  We now compute $\text{Ker}(\widetilde{Sq}^0_*)_{(5,10)}$.

\begin{lem}\label{bdday}
If $x$  is an admissible monomial of degree $25$ in $P_5$ and  $[x]$ belongs to  $\text{\rm Ker}(\widetilde{Sq}^0_*)_{(5,10)}$, then either $\omega(x) = (3,3,2,1)$ or $\omega(x) = (3,3,4)$. 
\end{lem}
\begin{proof}  Note that $z = x_1^{15}x_2^7x_3^3$ is the minimal spike of degree $25$ in $P_5$ and  $\omega(z) =  (3,3,2,1)$. Since $[x] \ne 0$, by Theorem \ref{dlsig}, either $\omega_1(x) =3$ or $\omega_1(x) =5$. If $\omega_1(x) =5$, then $x = X_\emptyset y^2$ with $y$ a monomial of degree $10$ in $P_5$. Since $x$ is admissible, by Theorem \ref{dlcb1}, $y$ is also admissible. Hence, $(\widetilde{Sq}^0_*)_{(5,10)}([x]) = [y] \ne 0.$ This contradicts the fact that $[x] \in \text{Ker}(\widetilde{Sq}^0_*)_{(5,10)}$, hence $\omega_1(x) =3$. Then, we have $x = x_ix_jx_\ell y_1^2$ with $1 \leqslant i < j < \ell \leqslant 5$ and $y_1$ an admissible monomial of degree $11$ in $P_5$. According to T\'in \cite{tin0}, either $\omega(y_1) = (3,2,1)$ or $\omega(y_1) = (3,4)$. The lemma follows.
\end{proof}

By Lemma \ref{bdday}, we have $\text{\rm Ker}(\widetilde{Sq}^0_*)_{(5,10)}\cong QP_5(3,3,2,1)\oplus QP_5(3,3,4).$

\begin{prop}\label{md334} $QP_5(3,3,4) = 0$.
\end{prop}

To prove this proposition, we need some lemmas. By a simple computation one gets the two following lemmas.

\begin{lem}\label{ina3341} If $x$ is one of the following monomials, then $f_i(x)$, $ 1 \leqslant i \leqslant 5$, are strictly inadmissible:

\medskip
\centerline{\begin{tabular}{llllll}
$x_1^{2}x_2x_3^{3}x_4^{3}$ &  $x_1^{2}x_2^{3}x_3x_4^{3}$ &  $x_1^{2}x_2^{3}x_3^{3}x_4$ &  $x_1^{3}x_2^{2}x_3x_4^{3}$ &  $x_1^{3}x_2^{2}x_3^{3}x_4$&  
$x_1^{3}x_2^{3}x_3^{2}x_4$ 
\end{tabular}}
\end{lem}

\begin{lem}\label{ina3342} The following monomials are strictly inadmissible:

\medskip
\centerline{\begin{tabular}{lllll}
$x_1^{2}x_2x_3x_4^{2}x_5^{3}$ &  $x_1^{2}x_2x_3x_4^{3}x_5^{2}$ &  $x_1^{2}x_2x_3^{2}x_4x_5^{3}$ &  $x_1^{2}x_2x_3^{2}x_4^{3}x_5$ &  $x_1^{2}x_2x_3^{3}x_4x_5^{2}$\cr  
$x_1^{2}x_2x_3^{3}x_4^{2}x_5$ &  $x_1^{2}x_2^{2}x_3x_4x_5^{3}$ &  $x_1^{2}x_2^{2}x_3x_4^{3}x_5$ &  $x_1^{2}x_2^{2}x_3^{3}x_4x_5$ &  $x_1^{2}x_2^{3}x_3x_4x_5^{2}$\cr  
$x_1^{2}x_2^{3}x_3x_4^{2}x_5$ &  $x_1^{2}x_2^{3}x_3^{2}x_4x_5$ &  $x_1^{3}x_2^{2}x_3x_4x_5^{2}$ &  $x_1^{3}x_2^{2}x_3x_4^{2}x_5$ &  $x_1^{3}x_2^{2}x_3^{2}x_4x_5$\cr 
\end{tabular}}
\end{lem}

\begin{lem}\label{ina3343} All permutations of the following monomials are strictly inadmissible:

\medskip
\centerline{\begin{tabular}{lll}
$x_1^{3}x_2^{4}x_3^{4}x_4^{7}x_5^{7}$ &  $x_1^{3}x_2^{4}x_3^{5}x_4^{6}x_5^{7}$ &  $x_1^{3}x_2^{5}x_3^{5}x_4^{6}x_5^{6}$ 
\end{tabular}}
\end{lem}
\begin{proof} We prove the lemma for the monomial $x = x_1^{3}x_2^{4}x_3^{4}x_4^{7}x_5^{7}$. The others can be proved by the similar computations. By a direct computation, we have
\begin{align*}
x &= Sq^1(x_1^{5}x_2x_3^{4}x_4^{7}x_5^{7}) + Sq^2(x_1^{3}x_2^{2}x_3^{4}x_4^{7}x_5^{7} + x_1^{6}x_2x_3^{2}x_4^{7}x_5^{7}) \ \text{ mod}\big(P_5^-(3,3,4)\big).
\end{align*}
This equality shows that all permutations of $x$ are strictly inadmissible.
\end{proof}
\begin{proof}[Proof of Proposition \ref{md334}] Let $x$ be an admissible monomial in $P_5$ such that $\omega(x) = (3,3,4)$. Then $x = x_jx_\ell x_ty^2$ with $y \in B_5(3,4)$ and $1 \leqslant j <\ell < t\leqslant 5$. 
Let $z \in B_5(3,4)$ such that $x_jx_\ell x_tz^2 \in P_5^+$.
By a direct computation we see that if $x_jx_\ell x_tz^2$ is not a permutation of one of monomials as given in Lemma \ref{ina3343}, then there is a monomial $w$ which is given in one of Lemmas \ref{ina3341}, \ref{ina3342} such that $x_jx_\ell x_tz^2= wz_1^{2^{u}}$ with suitable monomial $z_1 \in P_5$, and $u = \max\{j \in \mathbb Z : \omega_j(w) >0\}$. By Theorem \ref{dlcb1}, $x_jx_\ell x_tz^2$ is inadmissible. Since $x = x_jx_\ell x_ty^2$ and $x$ is admissible, $x$ is a permutation of one of monomials as given in Lemma \ref{ina3343}. Now the proposition follows from Lemma \ref{ina3343}. 
\end{proof}

By computing from a result in \cite{su5}, we have 
$$B_5^0(25) = B_5^0(3,3,2,1) = \{a_{25,t}:  1\leqslant t \leqslant 520\},$$ 
where the monomials $a_{25,t}$, $1 \leqslant t \leqslant 520$, are listed in Subsection \ref{s62}.

\begin{prop}\label{mdd61} There exist exactly $440$ admissible monomials in $P_5^+$ such that  their weight vectors are $(3,3,2,1)$. Consequently $\dim QP_5^+(3,3,2,1)= 440.$ 
\end{prop}

We prove the proposition by showing that $B_5^+(3,3,2,1) = \{b_{25,u}:  1\leqslant u \leqslant 440\},$
where the monomials $b_{25,u}$, $1\leqslant u \leqslant 440$, are listed in Subsection \ref{s63}.  

We need some lemmas for the proof of this proposition.
The following is a corollary of a result in \cite{su5}.
\begin{lem}\label{bd33211} If $x$ is one of the following monomials then $f_i(x)$, $ 1 \leqslant i \leqslant 5$, are strictly inadmissible:

\medskip
\centerline{\begin{tabular}{llll}
$x_1^{3}x_2^{4}x_3^{3}x_4^{7}$ &  $x_1^{7}x_2^{9}x_3^{2}x_4^{7}$ &  $x_1^{7}x_2^{9}x_3^{7}x_4^{2}$ &  $x_1^{7}x_2^{9}x_3^{6}x_4^{3}$\cr  
$x_1^{3}x_2^{7}x_3^{8}x_4^{7}$ &  $x_1^{7}x_2^{3}x_3^{8}x_4^{7}$ &  $x_1^{7}x_2^{8}x_3^{3}x_4^{7}$ &  $x_1^{7}x_2^{8}x_3^{7}x_4^{3}$\cr
\end{tabular}}
\end{lem}

\begin{lem}\label{bd33212} The following monomials are strictly inadmissible:

\medskip
\centerline{\begin{tabular}{lllll}
$x_1x_2^{2}x_3^{6}x_4x_5^{7}$ &  $x_1x_2^{2}x_3^{6}x_4^{7}x_5$ &  $x_1x_2^{2}x_3^{7}x_4^{6}x_5$ &  $x_1x_2^{6}x_3^{2}x_4x_5^{7}$ &  $x_1x_2^{6}x_3^{2}x_4^{7}x_5$\cr  
$x_1x_2^{6}x_3^{7}x_4^{2}x_5$ &  $x_1x_2^{7}x_3^{2}x_4^{6}x_5$ &  $x_1x_2^{7}x_3^{6}x_4^{2}x_5$ &  $x_1^{7}x_2x_3^{2}x_4^{6}x_5$ &  $x_1^{7}x_2x_3^{6}x_4^{2}x_5$\cr  
$x_1x_2^{6}x_3^{3}x_4^{6}x_5$ &  $x_1x_2^{6}x_3^{6}x_4x_5^{3}$ &  $x_1x_2^{6}x_3^{6}x_4^{3}x_5$ &  $x_1x_2^{2}x_3^{2}x_4^{5}x_5^{7}$ &  $x_1x_2^{2}x_3^{2}x_4^{7}x_5^{5}$\cr  
$x_1x_2^{2}x_3^{7}x_4^{2}x_5^{5}$ &  $x_1x_2^{7}x_3^{2}x_4^{2}x_5^{5}$ &  $x_1^{7}x_2x_3^{2}x_4^{2}x_5^{5}$ &  $x_1x_2^{2}x_3^{6}x_4^{3}x_5^{5}$ &  $x_1x_2^{2}x_3^{6}x_4^{5}x_5^{3}$\cr 
$x_1x_2^{6}x_3^{2}x_4^{3}x_5^{5}$ &  $x_1x_2^{6}x_3^{2}x_4^{5}x_5^{3}$ &  $x_1x_2^{6}x_3^{3}x_4^{2}x_5^{5}$ &  $x_1x_2^{6}x_3^{3}x_4^{4}x_5^{3}$ &  $x_1^{3}x_2^{4}x_3x_4^{6}x_5^{3}$\cr  
$x_1^{3}x_2^{4}x_3^{3}x_4^{3}x_5^{4}$ &  $x_1^{3}x_2^{4}x_3^{3}x_4^{4}x_5^{3}$ &
\end{tabular}}
\end{lem}
 \begin{proof} We prove this lemma for the monomial $x = x_1x_2^{2}x_3^{6}x_4^{7}x_5$. The others can be proved by the similar computations. We have
\begin{align*}
x &=  x_1^{}x_2^2x_3^{5}x_4^{7}x_5^{2} + x_1^{}x_2x_3^{6}x_4^{7}x_5^{2} + Sq^1(x_1^{2}x_2x_3^{5}x_4^{7}x_5^{}+x_1^{2}x_2x_3^{3}x_4^{9}x_5^{})\\
&\quad  + Sq^2(x_1^{}x_2^{}x_3^{5}x_4^{7}x_5^{} + x_1^{}x_2x_3^{3}x_4^{9}x_5^{}) \ \text{ mod}\big(P_5^-(3,3,2,1)\big).
\end{align*}
This equality shows that $x$ is strictly inadmissible.
\end{proof}

\begin{lem}\label{bd33213} The following monomials are strictly inadmissible:

\medskip
\centerline{\begin{tabular}{lllll}
$x_1x_2^{2}x_3^{3}x_4^{6}x_5^{13}$ &  $x_1x_2^{2}x_3^{3}x_4^{14}x_5^{5}$ &  $x_1x_2^{3}x_3^{2}x_4^{6}x_5^{13}$ &  $x_1x_2^{3}x_3^{2}x_4^{14}x_5^{5}$ &  $x_1x_2^{3}x_3^{6}x_4^{2}x_5^{13}$\cr  $x_1x_2^{3}x_3^{6}x_4^{6}x_5^{9}$ &  $x_1x_2^{3}x_3^{6}x_4^{10}x_5^{5}$ &  $x_1x_2^{3}x_3^{6}x_4^{14}x_5$ &  $x_1x_2^{3}x_3^{14}x_4^{2}x_5^{5}$ &  $x_1x_2^{3}x_3^{14}x_4^{6}x_5$\cr  $x_1x_2^{6}x_3^{3}x_4^{8}x_5^{7}$ &  $x_1x_2^{6}x_3^{8}x_4^{3}x_5^{7}$ &  $x_1x_2^{6}x_3^{8}x_4^{7}x_5^{3}$ &  $x_1x_2^{6}x_3^{9}x_4^{2}x_5^{7}$ &  $x_1x_2^{6}x_3^{9}x_4^{6}x_5^{3}$\cr  $x_1x_2^{6}x_3^{9}x_4^{7}x_5^{2}$ &  $x_1x_2^{7}x_3^{8}x_4^{2}x_5^{7}$ &  $x_1x_2^{7}x_3^{8}x_4^{6}x_5^{3}$ &  $x_1x_2^{7}x_3^{8}x_4^{7}x_5^{2}$ &  $x_1x_2^{7}x_3^{10}x_4^{4}x_5^{3}$\cr  $x_1^{3}x_2x_3^{2}x_4^{6}x_5^{13}$ &  $x_1^{3}x_2x_3^{2}x_4^{14}x_5^{5}$ &  $x_1^{3}x_2x_3^{6}x_4^{2}x_5^{13}$ &  $x_1^{3}x_2x_3^{6}x_4^{6}x_5^{9}$ &  $x_1^{3}x_2x_3^{6}x_4^{10}x_5^{5}$\cr  $x_1^{3}x_2x_3^{6}x_4^{14}x_5$ &  $x_1^{3}x_2x_3^{14}x_4^{2}x_5^{5}$ &  $x_1^{3}x_2x_3^{14}x_4^{6}x_5$ &  $x_1^{3}x_2^{3}x_3^{4}x_4^{4}x_5^{11}$ &  $x_1^{3}x_2^{3}x_3^{4}x_4^{12}x_5^{3}$\cr  $x_1^{3}x_2^{3}x_3^{12}x_4^{4}x_5^{3}$ &  $x_1^{3}x_2^{4}x_3x_4^{10}x_5^{7}$ &  $x_1^{3}x_2^{4}x_3^{9}x_4^{2}x_5^{7}$ &  $x_1^{3}x_2^{4}x_3^{9}x_4^{7}x_5^{2}$ &  $x_1^{3}x_2^{5}x_3^{2}x_4^{2}x_5^{13}$\cr  $x_1^{3}x_2^{5}x_3^{2}x_4^{6}x_5^{9}$ &  $x_1^{3}x_2^{5}x_3^{2}x_4^{10}x_5^{5}$ &  $x_1^{3}x_2^{5}x_3^{2}x_4^{14}x_5$ &  $x_1^{3}x_2^{5}x_3^{6}x_4^{2}x_5^{9}$ &  $x_1^{3}x_2^{5}x_3^{6}x_4^{10}x_5$\cr  $x_1^{3}x_2^{5}x_3^{8}x_4^{6}x_5^{3}$ &  $x_1^{3}x_2^{5}x_3^{10}x_4^{2}x_5^{5}$ &  $x_1^{3}x_2^{5}x_3^{10}x_4^{6}x_5$ &  $x_1^{3}x_2^{5}x_3^{14}x_4^{2}x_5$ &  $x_1^{3}x_2^{12}x_3x_4^{2}x_5^{7}$\cr  $x_1^{3}x_2^{12}x_3x_4^{7}x_5^{2}$ &  $x_1^{3}x_2^{12}x_3^{7}x_4x_5^{2}$ &  $x_1^{3}x_2^{13}x_3^{2}x_4^{2}x_5^{5}$ &  $x_1^{3}x_2^{13}x_3^{2}x_4^{6}x_5$ &  $x_1^{3}x_2^{13}x_3^{6}x_4^{2}x_5$\cr  $x_1^{7}x_2x_3^{8}x_4^{2}x_5^{7}$ &  $x_1^{7}x_2x_3^{8}x_4^{6}x_5^{3}$ &  $x_1^{7}x_2x_3^{8}x_4^{7}x_5^{2}$ &  $x_1^{7}x_2x_3^{10}x_4^{4}x_5^{3}$ &  $x_1^{7}x_2^{8}x_3x_4^{2}x_5^{7}$\cr  $x_1^{7}x_2^{8}x_3x_4^{6}x_5^{3}$ &  $x_1^{7}x_2^{8}x_3x_4^{7}x_5^{2}$ &  $x_1^{7}x_2^{8}x_3^{3}x_4^{3}x_5^{4}$ &  $x_1^{7}x_2^{8}x_3^{7}x_4x_5^{2}$ &  $x_1^{7}x_2^{9}x_3^{2}x_4^{4}x_5^{3}$\cr  $x_1^{7}x_2^{9}x_3^{6}x_4x_5^{2}$ &
\end{tabular}}
\end{lem}
 \begin{proof} We prove this lemma for the monomial $x = x_1^3x_2^{5}x_3^{8}x_4^{6}x_5^3$. The others can be proved by the similar computations. We have
\begin{align*}
x &= x_1^{2}x_2^{3}x_3^{5}x_4^{6}x_5^{9} +  x_1^{2}x_2^{3}x_3^{5}x_4^{9}x_5^{6} +  x_1^{2}x_2^{3}x_3^{6}x_4^{5}x_5^{9} +  x_1^{2}x_2^{3}x_3^{6}x_4^{9}x_5^{5} +  x_1^{2}x_2^{3}x_3^{9}x_4^{5}x_5^{6}\\  &\quad + x_1^{2}x_2^{3}x_3^{9}x_4^{6}x_5^{5} +  x_1^{2}x_2^{5}x_3^{3}x_4^{6}x_5^{9} +  x_1^{2}x_2^{5}x_3^{3}x_4^{9}x_5^{6} +  x_1^{2}x_2^{5}x_3^{6}x_4^{3}x_5^{9} +  x_1^{2}x_2^{5}x_3^{6}x_4^{9}x_5^{3}\\  &\quad + x_1^{2}x_2^{5}x_3^{9}x_4^{3}x_5^{6} +  x_1^{2}x_2^{5}x_3^{9}x_4^{6}x_5^{3} +  x_1^{3}x_2^{3}x_3^{4}x_4^{6}x_5^{9} +  x_1^{3}x_2^{3}x_3^{4}x_4^{9}x_5^{6} +  x_1^{3}x_2^{3}x_3^{5}x_4^{6}x_5^{8}\\  &\quad + x_1^{3}x_2^{3}x_3^{5}x_4^{8}x_5^{6} +  x_1^{3}x_2^{3}x_3^{6}x_4^{4}x_5^{9} +  x_1^{3}x_2^{3}x_3^{6}x_4^{5}x_5^{8} +  x_1^{3}x_2^{3}x_3^{6}x_4^{8}x_5^{5} +  x_1^{3}x_2^{3}x_3^{6}x_4^{9}x_5^{4}\\  &\quad + x_1^{3}x_2^{3}x_3^{8}x_4^{5}x_5^{6} +  x_1^{3}x_2^{3}x_3^{8}x_4^{6}x_5^{5} +  x_1^{3}x_2^{3}x_3^{9}x_4^{4}x_5^{6} +  x_1^{3}x_2^{3}x_3^{9}x_4^{6}x_5^{4} +  x_1^{3}x_2^{4}x_3^{3}x_4^{6}x_5^{9}\\  &\quad + x_1^{3}x_2^{4}x_3^{3}x_4^{9}x_5^{6} +  x_1^{3}x_2^{4}x_3^{6}x_4^{3}x_5^{9} +  x_1^{3}x_2^{4}x_3^{6}x_4^{9}x_5^{3} +  x_1^{3}x_2^{4}x_3^{9}x_4^{3}x_5^{6} +  x_1^{3}x_2^{4}x_3^{9}x_4^{6}x_5^{3}\\  &\quad + x_1^{3}x_2^{5}x_3^{3}x_4^{6}x_5^{8} +  x_1^{3}x_2^{5}x_3^{3}x_4^{8}x_5^{6} +  x_1^{3}x_2^{5}x_3^{6}x_4^{3}x_5^{8} +  x_1^{3}x_2^{5}x_3^{6}x_4^{8}x_5^{3} +  x_1^{3}x_2^{5}x_3^{8}x_4^{3}x_5^{6}\\  
&\quad + Sq^1\big(x_1^{3}x_2^{3}x_3^{3}x_4^{6}x_5^{9} +  x_1^{3}x_2^{3}x_3^{3}x_4^{9}x_5^{6} +  x_1^{3}x_2^{3}x_3^{6}x_4^{3}x_5^{9} +  x_1^{3}x_2^{3}x_3^{6}x_4^{9}x_5^{3} +  x_1^{3}x_2^{3}x_3^{9}x_4^{3}x_5^{6}\\  
&\quad + x_1^{3}x_2^{3}x_3^{9}x_4^{6}x_5^{3} +  x_1^{3}x_2^{6}x_3^{5}x_4^{5}x_5^{5}\big) + Sq^2\big(x_1^{2}x_2^{3}x_3^{3}x_4^{6}x_5^{9} +  x_1^{2}x_2^{3}x_3^{3}x_4^{9}x_5^{6} + x_1^{2}x_2^{3}x_3^{6}x_4^{3}x_5^{9}\\  
&\quad + x_1^{2}x_2^{3}x_3^{6}x_4^{9}x_5^{3} +  x_1^{2}x_2^{3}x_3^{9}x_4^{3}x_5^{6} +  x_1^{2}x_2^{3}x_3^{9}x_4^{6}x_5^{3} +  x_1^{3}x_2^{5}x_3^{5}x_4^{5}x_5^{5} +  x_1^{3}x_2^{5}x_3^{5}x_4^{5}x_5^{5}\\  
&\quad + x_1^{5}x_2^{3}x_3^{3}x_4^{6}x_5^{6} +  x_1^{5}x_2^{3}x_3^{5}x_4^{5}x_5^{5} +  x_1^{5}x_2^{3}x_3^{6}x_4^{3}x_5^{6} +  x_1^{5}x_2^{3}x_3^{6}x_4^{6}x_5^{3} +  x_1^{5}x_2^{5}x_3^{3}x_4^{5}x_5^{5}\\  
&\quad + x_1^{5}x_2^{5}x_3^{5}x_4^{3}x_5^{5}\big) + Sq^4\big(x_1^{3}x_2^{3}x_3^{3}x_4^{3}x_5^{9} +  x_1^{3}x_2^{3}x_3^{3}x_4^{6}x_5^{6} +  x_1^{3}x_2^{3}x_3^{6}x_4^{3}x_5^{6}\\  
&\quad + x_1^{3}x_2^{3}x_3^{6}x_4^{6}x_5^{3}\big) + Sq^8\big(x_1^{3}x_2^{3}x_3^{3}x_4^{3}x_5^{5}\big)\ \text{ mod}\big(P_5^-(3,3,2,1)\big).
\end{align*}
This equality implies that $x$ is strictly inadmissible.
\end{proof}

\begin{proof}[Proof of Proposition \ref{mdd61}] Let $x$ be an admissible monomial in $(P_5^+)_{25}$ such that  $\omega(x) = (3,3,2,1)$. Then, $x = x_jx_\ell x_ty^2$ with $1 \leqslant j < \ell < t \leqslant 5$ and $y \in B_5(3,2,1)$. 

Let $z \in B_5(3,2,1)$ such that $x_jx_\ell x_tz^2 \in P_5^+$.
By a direct computation using the results in T\'in \cite{tin0}, we see that if $x_jx_\ell x_tz^2 \ne b_{25,u}, \forall u, \ 1 \leqslant u \leqslant 440$, then there is a monomial $w$ which is given in one of Lemmas \ref{ina3341}, \ref{bd33211}, \ref{bd33212} and \ref{bd33213} such that $x_jx_\ell x_tz^2= wz_1^{2^{r}}$ with suitable monomial $z_1 \in P_5$, and $r = \max\{j \in \mathbb Z : \omega_j(w) >0\}$. By Theorem \ref{dlcb1}, $x_jx_\ell x_tz^2$ is inadmissible. Since $x$ is admissible and $x = x_jx_\ell x_ty^2$ with $y \in B_5(3,2,1)$, one gets $x= b_{25,u}$ for some $u,\ 1 \leqslant u \leqslant 440$. This implies $B_5^+(3,3,2,1) \subset \{b_{25,u} : \ 1 \leqslant u \leqslant 440\}$. 

We now prove the set $\{[b_{25,u}] : \ 1 \leqslant u \leqslant 440\}$ is linearly independent in the space $(\mathbb F_2 \otimes_{\mathcal A} P_5)_{25}$. Suppose there is a linear relation
$$\mathcal S = \sum_{u = 1}^{440}\gamma_ub_{25,u} \equiv 0,$$ 
where $\gamma_u \in \mathbb F_2$. For $1 \leqslant i < j \leqslant 5$, we explicitly compute $p_{(i;I)}(\mathcal S)$ in terms of the admissible monomials in $P_4$ (mod$(\mathcal A^+P_4)$). By a direct computation from the relations $p_{(i;I)}(\mathcal S) \equiv 0$ with $\ell(I) \leqslant 2$, we obtain $\gamma_u = 0$ for $1 \leqslant u \leqslant 440$.
The proposition follows. 
\end{proof}

We have $\dim (QP_5)_{10} = 280$ and $\dim (QP_5^0)_{25}  = 520$. So, one gets the following theorem.
\begin{thm}\label{dlad1} For any positive integer $s$, we have $\dim (QP_5)_{15.2^s - 5} = 1240$. 
\end{thm}
\section{An application to the fifth Singer transfer}\label{s5}

In this section, we verify Singer's conjecture for the fifth algebraic tranfer in the internal degree $15.2^s-5$ by using the results in Section \ref{s4}. The main result of the section is the following.

\begin{thm}\label{dlc5} For any non-negative integer $s$, we have $(QP_5)_{15.2^s-5}^{GL_5} = 0$.
\end{thm}

By using the results of Tangora \cite{ta}, Lin \cite{wl} and Chen \cite{che}, we easily obtain 
$\text{Tor}_{5,7.2^s}^{\mathcal A}(\mathbb F_2,\mathbb F_2) = 0.$
So, by Theorem \ref{dlc5}, the homomorphism
$$\varphi_5: \text{Tor}_{5,15.2^s}^{\mathcal A}(\mathbb F_2,\mathbb F_2) \longrightarrow (\mathbb F_2{\otimes}_{\mathcal A}P_5)_{15.2^s-5}^{GL_5}$$
is a trivial isomorphism. Hence, one gets the following.
\begin{corl} Singer's conjecture is true for the case $n = 5$ and the internal degree $15.2^s-5$ with $s$ an arbitrary non-negative integer.
\end{corl} 
Recall that the iterated Kameko homomorphism
$$(\widetilde{Sq}^0_*)^{s-1}: (QP_5)_{15.2^s-5} \longrightarrow (QP_5)_{25}$$
is an isomorphism of $GL_5$-modules for every $s \geqslant 1$. Hence, we need only to prove Theorem \ref{dlc5} for $s = 0,1$. 

We need a notation for the proof of the theorem. For a fixed weight vector $\omega$, for any monomials $z_1, z_2, \ldots, z_m$ in $P_n(\omega)$ and for a subgroup $G\subset GL_n$, we denote 
$G(z_1, z_2, \ldots, z_m)$ the $G$-submodule of $QP_n(\omega)$ generated by the set $\{[z_i]_\omega : 1 \leqslant i \leqslant m\}$. 

\subsection{The case $s = 0$}\label{s51}\

\medskip
Denote by $a_t = a_{10,t}$, $1 \leqslant t \leqslant 280$, the admissible monomials of degree 10 as given in Subsection \ref{s61}.
We have 
$$(QP_5)_{10} \cong QP_5(2,2,1)\oplus QP_5(2,4)\oplus QP_5(4,1,1)\oplus QP_5(4,3).$$

We need some lemmas.

\begin{lems}\label{bd54} For $\omega = (2,2,1)$, we have $\dim(QP_5^0(\omega))^{\Sigma_5} = 4.$
\end{lems}
\begin{proof} Observe that $\omega = (2,2,1)$ is the weight vector of the mimimal spike $x_1^7x_2^3$, hence $[f]_\omega = [f]$ for all $f \in P_5$. We have a direct summand decomposition of the $\Sigma_5$-modules
$$QP_5^0(\omega) = \Sigma_5(a_{1})\oplus \Sigma_5(a_{21})\oplus \Sigma_5(a_{51})\oplus \Sigma_5(a_{116}).$$ 
We prove the following:
\begin{align}\Sigma_5(a_{1})^{\Sigma_5}\label{ct51} &= \langle [p_1 := a_1 + a_2 + \ldots + a_{20}] \rangle,\\
\Sigma_5(a_{21})^{\Sigma_5}\label{ct52} &= \langle [p_2 := a_{21} + a_{22} \ldots + a_{50}] \rangle,\\
\Sigma_5(a_{51})^{\Sigma_5} &= \langle [p_3 := a_{51} + a_{52} + \dots + a_{90}] \rangle,\\
\Sigma_5(a_{126})^{\Sigma_5} &= \langle [p_4 := a_{101} + a_{102}+ \dots + a_{145}] \rangle.
\end{align}
For simplicity, we prove \eqref{ct52} in detail. The others can be proved by a similar computation. It is easy to see that 
$\Sigma_5(a_{21}) = \langle [a_t] : 21 \leqslant t \leqslant 50\rangle.$
 Let $[g]\in \Sigma_5(a_{21})^{\Sigma_5}$ with  $g  = \sum_{t=21}^{50}\gamma_ta_t$ and $\gamma_t \in \mathbb F_2$. Computing directly $\rho_j(g) + g$ in terms of $a_t$, $21 \leqslant t \leqslant 50$ (mod($P_5^-(\omega))$. From the relations $\rho_j(g) + g \equiv 0$ with $1 \leqslant j \leqslant 4$, we get $\gamma_t = \gamma_1$ for $21 \leqslant t \leqslant 50$. The lemma follows.
\end{proof}
By an argument analogous to the previous one, we easily obtain the following.
\begin{lems}\label{bd53} Denote $M =  \langle [a_t] : 231 \leqslant t \leqslant 240\rangle$. Then, $M$ is an $\Sigma_5$-submodule of $(QP_5^+)_{10}$ and $M^{\Sigma_5} = \langle [p_5 :=a_{232} + a_{233} + a_{234} + a_{235}]\rangle.$
\end{lems}

\begin{lems}\label{bd52} For $\omega = (4,1,1)$, we have $(QP_5(\omega))^{GL_5} = \langle [p_6]_{\omega}\rangle$, where
$$p_6 = a_{241} + a_{242} + a_{243} + a_{244} + a_{257} + a_{258} + a_{259} + a_{260} .$$
\end{lems}
\begin{proof} We have $QP_5(\omega) = \Sigma_5(a_{161})\oplus \Sigma_5(a_{181})\oplus \Sigma_5(a_{241})$. A direct computation shows 
\begin{align*} 
\Sigma_5(a_{241}) &= \langle [a_t]_\omega  : 241 \leqslant t \leqslant 260\rangle,\ \Sigma_5(a_{211})^{\Sigma_5} = \langle [p_6]_\omega  \rangle, \\ 
\Sigma_5(a_{161}) &= \langle [a_t]_\omega  : 161 \leqslant t \leqslant 180\rangle,\ \Sigma_5(a_{211})^{\Sigma_5} = \langle [p_7]_\omega  \rangle, \\
\Sigma_5(a_{181}) &= \langle [a_t]_\omega  : 181 \leqslant t \leqslant 210\rangle,\ \Sigma_5(a_{261})^{\Sigma_5} = \langle [p_8]_\omega  \rangle,
\end{align*}
where $p_7 = \sum_{t=161}^{180} a_t,\ p_8 = \sum_{t=181}^{210} a_t$. Suppose $[f]_\omega \in (QP_5(\omega))^{GL_5}$ with  $f \in P_5(\omega)$. Since $[f]_\omega \in (QP_5(\omega))^{\Sigma_5}$, we obtain, $f \equiv_\omega \gamma_1p_6 + \gamma_2p_7 + \gamma_3p_8$ with $\gamma_1, \gamma_2, \gamma_3 \in \mathbb F_2$. By a direct computation, we get
$$\rho_5(f)+f \equiv_\omega \gamma_2a_{161} + \gamma_3a_{181} + \text{ other terms} \equiv_\omega 0.$$
From this we get $\gamma_2 = \gamma_3 = 0$. The lemma is proved. 
\end{proof}

\begin{lems}\label{bd51} $(QP_5(4,3))^{GL_5} = 0.$
\end{lems}
\begin{proof} It is easy to see that $QP_5(\omega) = \Sigma_5(a_{211})\oplus \Sigma_5(a_{261})$ with $\omega = (4,3)$. By a simple computation, we have 
\begin{align*} \Sigma_5(a_{211}) &= \langle [a_t]_\omega : 211 \leqslant t \leqslant 230\rangle,\ \Sigma_5(a_{211})^{\Sigma_5} = \langle [p_9]_\omega \rangle, \\
\Sigma_5(a_{261}) &= \langle [a_t]_\omega : 261 \leqslant t \leqslant 280\rangle,\ \Sigma_5(a_{261})^{\Sigma_5} = \langle [p_{10}]_\omega \rangle, 
\end{align*}
where $p_9 = \sum_{t=211}^{230} a_t,\ p_{10} = \sum_{t=261}^{280} a_t$. Let $[f]_\omega \in (QP_5(4,3))^{GL_5}$ with  $f \in P_5(4,3)$, then $[f]_\omega \in (QP_5(4,3))^{\Sigma_5}$. Hence, $f \equiv_\omega \gamma_1p_9 + \gamma_2p_{10}$ with $\gamma_1, \gamma_2 \in \mathbb F_2$. By a direct computation, we get
$$\rho_5(f)+f \equiv_\omega \gamma_1a_{211} + \gamma_2a_{267} + \text{ other terms} \equiv_\omega 0.$$
This equality implies $\gamma_1 = \gamma_2 = 0$. The lemma follows. 
\end{proof}

\begin{proof}[Proof of Theoren \ref{dlc5} for $s = 0$] Let $[f] \in (QP_5)_{10}^{GL_5}$. From Lemmas \ref{bd51} and \ref{bd52}, we have $f \equiv f' + \gamma_6 p_6$ with $f' \in P_5^-(4,1,1)$. A simple computation shows that $[p_6] \in (QP_5)_{10}^{\Sigma_5}$, hence $[f']$ is an $\Sigma_5$-invariant. Since $[P_5^-(4,1,1)] = (QP_5^0(2,2,1) \oplus M$, using Lemmas \ref{bd54} and \ref{bd53}, we obtain $f' \equiv \sum_{j=1}^5\gamma_jp_j$ with $\gamma_j \in \mathbb F_2$. By computing $\rho_5(f)+f$ in terms of the admissible monomials, we get
\begin{align*}\rho_5(f)+f &\equiv \gamma_1a_{15} + \gamma_2a_{24} + \gamma_3a_{56} + \gamma_4a_{92}\\
&\hskip1cm + (\gamma_4 + \gamma_5)a_{119} + \gamma_6a_{239} + \text{ other terms} \equiv 0.
\end{align*}
From this it implies $\gamma_j =0$, $1 \leqslant j \leqslant 6$. Theorem \ref{dlc5} is true for $s = 0$.
\end{proof}

\subsection{The case $s = 1$}\label{s52}\

\medskip
Recall that Kameko's homomorphism
$(\widetilde{Sq}^0_*)_{(5,10)}: (QP_5)_{25} \longrightarrow (QP_5)_{10}$
is an epimorphism of $GL_5$-modules, hence from Theorem \ref{dlc5} for $s = 0$, we need only to compute $\text{Ker}(\widetilde{Sq}^0_*)_{(5,10)}^{GL_5}$. 

 From the results in Section \ref{s4}, we see that 
$$\text{\rm Ker}(\widetilde{Sq}^0_*)_{(5,10)} = QP_5(3,3,2,1) = QP_5^0(3,3,2,1) \oplus QP_5^+(3,3,2,1)$$ 
where $\dim QP_5^0(3,3,2,1) = 520$ and $\dim QP_5^+(3,3,2,1) = 440.$ Denote by $a_t = a_{25,t}$, $1 \leqslant t \leqslant 520$, the admissible monomials in $(P_5^0)_{25}$ as given in Subsection \ref{s62}.
By a direct computation, we can see that
\begin{align*}
\Sigma_5(a_1) &= \langle [a_t] : 1 \leqslant t \leqslant 60\rangle,\\
\Sigma_5(a_{61}) &= \langle [a_t] : 61 \leqslant t \leqslant 80\rangle,\\
\Sigma_5(a_{81}) &= \langle [a_t] : 81 \leqslant t \leqslant 140\rangle,\\
\Sigma_5(a_{141}) &= \langle [a_t] : 141 \leqslant t \leqslant 240\rangle,
\end{align*}
and $M_1 = \langle [a_t] : 241 \leqslant t \leqslant 520\rangle$ are $\Sigma_5$-submodules of $QP_5^0(3,3,2,1)$. Then, we have a direct summand decomposition of $\Sigma_5$-modules:
$$QP_5^0(3,3,2,1) = \Sigma_5(a_{1})\oplus \Sigma_5(a_{61})\oplus \Sigma_5(a_{81})\oplus \Sigma_5(a_{141})\oplus M_1.$$ 

\begin{lems}\label{bd251} We have
\begin{align*} &\dim \Sigma_5(a_{1})^{\Sigma_5} = \dim\Sigma_5(a_{81})^{\Sigma_5} = \dim\Sigma_5(a_{141})^{\Sigma_5} =1,\\ 
&\Sigma_5(a_{61}^{\Sigma_5}) = 0, \ \dim M_1^{\Sigma_5} = 3.
\end{align*}
\end{lems}
\begin{proof}[Outline of the proof] The set $A := \{[a_t] : 1 \leqslant t \leqslant 60\}$ is a basis of $\langle[\Sigma_5(a_1)]\rangle$. The action of $\Sigma_5$ on $QP_5$ induces the one of it on $A$. Furthermore, this action is transitive. Hence, if $f = \sum_{t=1}^{60}\gamma_ta_t$ with $\gamma_t \in \mathbb F_2$ and $[f] \in  \langle[\Sigma_5(a_1)]\rangle^{\Sigma_5}$, then the relations $\rho_j(f) \equiv f, j = 1,2,3,4,$ imply $\gamma_t = \gamma_1, \forall t,\ 1 \leqslant t \leqslant 60$. Hence, $\Sigma_5(a_{1})^{\Sigma_5} = \langle [q_1]\rangle$ with $q_1 = \sum_{t=1}^{60}a_t$. Similarly, we have $\Sigma_5(a_{81})^{\Sigma_5} = \langle [q_2]\rangle$ with $q_2 = \sum_{t=81}^{140}a_t$.

Suppose $g = \sum_{t=141}^{240}a_t$ and $[g] \in \Sigma_5(a_{141})^{\Sigma_5}$. We compute $\rho_j(g)+g$ in terms of 
$a_t ,\ 141 \leqslant t \leqslant 240$. By computing directly from the relations $\rho_j(g)+g \equiv 0$, $j = 1,2,3,4$, we obtain $\gamma_t = 0$ for $221 \leqslant t \leqslant 240$ and $\gamma_t = \gamma_{141}$ for $141 \leqslant t \leqslant 220$. Hence  $\Sigma_5(a_{141})^{\Sigma_5} = \langle [q_3]\rangle$ with $q_3 = \sum_{t=141}^{220}a_t$. By a similar computation, we get $\Sigma_5(a_{61})^{\Sigma_5} = 0$.

Suppose $g = \sum_{t=241}^{520}\gamma_ta_t$ and $[g] \in M_1^{\Sigma_5}$. By computing $\rho_j(g)+g$ in terms of 
$a_t ,\ 241 \leqslant t \leqslant 520$ and using the relations $\rho_j(g)+g \equiv 0$, $j = 1,2,3,4$, one gets 
\begin{align}\label{ct57}
\begin{cases}
\gamma_t = 0, \text{ for }  486 \leqslant t \leqslant 520,\ \gamma_t = \gamma_{241}\text{ for } 241 \leqslant t \leqslant 330,\\ 
\gamma_t = \gamma_{331}\text{ for } 331 \leqslant t \leqslant 350,\ \gamma_t = \gamma_{351}\text{ for } 351 \leqslant t \leqslant 400,\\
\gamma_t = \gamma_{401}\text{ for } 401 \leqslant t \leqslant 425,\ \gamma_t = \gamma_{426}\text{ for } 426 \leqslant t \leqslant 465,\\
\gamma_t = \gamma_{466}\text{ for } 466 \leqslant t \leqslant 480,\ \gamma_t = \gamma_{351}\text{ for } 481 \leqslant t \leqslant 485,\\
\gamma_{241} +  \gamma_{351} + \gamma_{401} = \gamma_{241} +  \gamma_{331} + \gamma_{426} = 0,\\
\gamma_{241} +  \gamma_{331} + \gamma_{351} + \gamma_{466} = \gamma_{331} +  \gamma_{351} + \gamma_{481} =
0.
\end{cases}
\end{align}
Set $q_4 = \sum_{t=141}^{330}a_t + \sum_{t=401}^{480}a_t,$\ $q_5 = \sum_{t=331}^{350}a_t + \sum_{t=426}^{485}a_t,$\ $ q_6 =  \sum_{t=351}^{425}a_t + \sum_{t=466}^{485}a_t.$
From \eqref{ct57}, we obtain $M_1^{\Sigma_5} = \langle [q_4], [q_5], [q_6]\rangle$. The lemma follows.
\end{proof}

Denote by $b_u = b_{25,u}$, $1 \leqslant u \leqslant 440$, the admissible monomials in $(P_5^+)_{25}$ as given in Subsections \ref{s63}. It is easy to see that
$\Sigma_5(b_1) = \langle [b_u] : 1 \leqslant u \leqslant 60\rangle,$
$M_2 = \langle [b_u] : 60 \leqslant u \leqslant 440\rangle$ are $\Sigma_5$-submodules of $QP_5^+(3,3,2,1)$ and we have a direct summand decomposition of $\Sigma_5$-modules:
$QP_5^+(3,3,2,1) = \Sigma_5(b_{1})\oplus  M_2.$ 
\begin{lems}\label{bd252} We have
$\dim \Sigma_5(b_{1})^{\Sigma_5} = 1$ and $\dim M_2^{\Sigma_5} = 3.$
\end{lems}
\begin{proof}[Outline of the proof] 
Suppose $f = \sum_{u=1}^{60}b_u$ and $[f] \in \Sigma_5(b_{1})^{\Sigma_5}$. A direct computation from the relations $\rho_j(f)+f \equiv 0$, $j = 1,2,3,4$, shows that $\gamma_u = 0$ for $46 \leqslant u \leqslant 60$ and $\gamma_u = \gamma_{1}$ for $1 \leqslant u \leqslant 45$. Hence  $\Sigma_5(b_{1})^{\Sigma_5} = \langle [q_7]\rangle$ with $q_7 = \sum_{u=1}^{45}b_u$. 

Suppose $g = \sum_{u=61}^{440}\gamma_ub_u$ and $[g] \in M_2^{\Sigma_5}$. We compute $\rho_j(g)+g$ in terms of 
$b_u ,\ 61 \leqslant u \leqslant 440$. By using the relations $\rho_j(g)+g \equiv 0$, $j = 1,2,3,4$, we have 
\begin{align}\label{ct58}
\begin{cases}
\gamma_u = 0, \text{ for }  267 \leqslant u \leqslant 440,\ \gamma_u = \gamma_{61}\text{ for } 61 \leqslant u \leqslant 125,\\ 
\gamma_u = \gamma_{126}\text{ for } 126 \leqslant u \leqslant 175,\ \gamma_u = \gamma_{176}\text{ for } 176 \leqslant u \leqslant 226,\\
\gamma_u = \gamma_{227}\text{ for } 227 \leqslant u \leqslant 258,\ \gamma_u = \gamma_{259}\text{ for } 259 \leqslant u \leqslant 262,\\
\gamma_{263} =  \gamma_{264} = \gamma_{265},\
 \gamma_{61} +  \gamma_{126} + \gamma_{227} = \gamma_{126} +  \gamma_{176} + \gamma_{259} = 0,\\
\gamma_{61} +  \gamma_{126} + \gamma_{176} + \gamma_{263} = \gamma_{61} +  \gamma_{176} + \gamma_{266} =
0.
\end{cases}
\end{align}
Set $q_8 = \sum_{u=61}^{125}b_u + \sum_{u=227}^{258}b_u + \sum_{u=263}^{266}b_u,$\ $q_9 = \sum_{u=126}^{175}b_u + \sum_{u=226}^{265}b_u,$\ $ q_{10} =  \sum_{u=176}^{226}b_u + \sum_{u=259}^{266}b_u.$
From \eqref{ct58}, we obtain $M_2^{\Sigma_5} = \langle [q_8], [q_9], [q_{10}]\rangle$. The lemma is proved.
\end{proof}

Combining Lemmas \ref{bd251} and \ref{bd252}, one gets the following.
\begin{corls}$\dim\text{\rm Ker}(\widetilde{Sq}^0_*)_{(5,10)}^{\Sigma_5} = 10.$
\end{corls}

\begin{proof}[Proof of Theorem \ref{dlc5} for $s = 1$] Suppose $[h] \in \text{\rm Ker}(\widetilde{Sq}^0_*)_{(5,10)}^{GL_5}$ with $h$ a polynomial in $(P_5)_{25}$. Since $[h] \in \text{\rm Ker}(\widetilde{Sq}^0_*)_{(5,10)}^{\Sigma_5}$, from Lemmas \ref{bd251} and \ref{bd252}, we have $h \equiv \sum_{j=1}^{10}\delta_jq_j$ with $\delta_j \in \mathbb F_2$.  By computing $\rho_5(h)+h$ in terms of the admissible monomials, we get
\begin{align*}\rho_5(h)+h &\equiv \delta_{1}a_{31} + \delta_{2}a_{81} + \delta_{3}a_{144} + \delta_{4}a_{234} + (\delta_5 + \delta_{10})a_{331} + \delta_{6}a_{263}\\
&\quad + \delta_{7}b_{19} + \delta_{8}b_{78} + \delta_{9}b_{138} + \delta_{10}b_{187} + \text{ other terms} \equiv 0.
\end{align*}
The last equality implies that $\gamma_j =0$, $1 \leqslant j \leqslant 6$. The proof is completed.
\end{proof}

\section{Appendix}\label{s6}   

In this section, we list the admissible monomials of degrees 10, 11, 25 in $P_5$. 

\subsection{The admissible monomials of degree 10 in $P_5$}\label{s61}\

\medskip
$B_5^0(10)$ is the set of 230 monomials: $a_{10,t},\ 1 \leqslant t \leqslant 230$.

\medskip
\centerline{
\begin{tabular}{llll}
$1.\ x_4^{3}x_5^{7}$ &  $2.\ x_4^{7}x_5^{3}$ &  $3.\ x_3^{3}x_5^{7}$ &  $4.\ x_3^{3}x_4^{7}$\cr  
$5.\ x_3^{7}x_5^{3}$ &  $6.\ x_3^{7}x_4^{3}$ &  $7.\ x_2^{3}x_5^{7}$ &  $8.\ x_2^{3}x_4^{7}$\cr  
$9.\ x_2^{3}x_3^{7}$ &  $10.\ x_2^{7}x_5^{3}$ &  $11.\ x_2^{7}x_4^{3}$ &  $12.\ x_2^{7}x_3^{3}$\cr  
$13.\ x_1^{3}x_4^{7}$ &  $14.\ x_1^{3}x_3^{7}$ &  $15.\ x_1^{3}x_2^{7}$ &  $16.\ x_1^{7}x_5^{3}$\cr  
$17.\ x_1^{7}x_4^{3}$ &  $18.\ x_1^{7}x_3^{3}$ &  $19.\ x_1^{7}x_2^{3}$ &  $20.\ x_1^{3}x_5^{7}$\cr  
$21.\ x_3x_4^{2}x_5^{7}$ &  $22.\ x_3x_4^{7}x_5^{2}$ &  $23.\ x_3^{7}x_4x_5^{2}$ &  $24.\ x_2x_4^{2}x_5^{7}$\cr  
$25.\ x_2x_4^{7}x_5^{2}$ &  $26.\ x_2x_3^{2}x_5^{7}$ &  $27.\ x_2x_3^{2}x_4^{7}$ &  $28.\ x_2x_3^{7}x_5^{2}$\cr  
$29.\ x_2x_3^{7}x_4^{2}$ &  $30.\ x_2^{7}x_4x_5^{2}$ &  $31.\ x_2^{7}x_3x_5^{2}$ &  $32.\ x_2^{7}x_3x_4^{2}$\cr
\end{tabular}}
\centerline{\begin{tabular}{llll}
$141.\ x_1x_2^{2}x_3^{4}x_5^{3}$ &  $142.\ x_1x_2^{2}x_3^{4}x_4^{3}$ &  $143.\ x_1x_2^{3}x_4^{2}x_5^{4}$ &  $144.\ x_1x_2^{3}x_4^{4}x_5^{2}$\cr  
$145.\ x_1x_2^{3}x_3^{2}x_5^{4}$ &  $146.\ x_1x_2^{3}x_3^{2}x_4^{4}$ &  $147.\ x_1x_2^{3}x_3^{4}x_5^{2}$ &  $148.\ x_1x_2^{3}x_3^{4}x_4^{2}$\cr  
$149.\ x_1^{3}x_3x_4^{2}x_5^{4}$ &  $150.\ x_1^{3}x_3x_4^{4}x_5^{2}$ &  $151.\ x_1^{3}x_3^{4}x_4x_5^{2}$ &  $152.\ x_1^{3}x_2x_4^{2}x_5^{4}$\cr  
$153.\ x_1^{3}x_2x_4^{4}x_5^{2}$ &  $154.\ x_1^{3}x_2x_3^{2}x_5^{4}$ &  $155.\ x_1^{3}x_2x_3^{2}x_4^{4}$ &  $156.\ x_1^{3}x_2x_3^{4}x_5^{2}$\cr  
$157.\ x_1^{3}x_2x_3^{4}x_4^{2}$ &  $158.\ x_1^{3}x_2^{4}x_4x_5^{2}$ &  $159.\ x_1^{3}x_2^{4}x_3x_5^{2}$ &  $160.\ x_1^{3}x_2^{4}x_3x_4^{2}$\cr  
$161.\ x_2x_3x_4x_5^{7}$ &  $162.\ x_2x_3x_4^{7}x_5$ &  $163.\ x_2x_3^{7}x_4x_5$ &  $164.\ x_2^{7}x_3x_4x_5$\cr  
$165.\ x_1x_3x_4x_5^{7}$ &  $166.\ x_1x_3x_4^{7}x_5$ &  $167.\ x_1x_3^{7}x_4x_5$ &  $168.\ x_1x_2x_4x_5^{7}$\cr  
$169.\ x_1x_2x_4^{7}x_5$ &  $170.\ x_1x_2x_3x_5^{7}$ &  $171.\ x_1x_2x_3x_4^{7}$ &  $172.\ x_1x_2x_3^{7}x_5$\cr
\end{tabular}}

\medskip
We have $B_5^+(10) = B_5^+(2,2,1)\cup B_5(2,4)\cup B_5^+(4,1,1) \cup B_5^+(4,3)$

\smallskip
$B_5^+(2,2,1)$ is the set of 5 monomials: $a_{10,t},\ 231 \leqslant t \leqslant 235$.

\medskip
\centerline{\begin{tabular}{llll}
$231.\ x_1x_2x_3^{2}x_4^{2}x_5^{4}$ &  $232.\ x_1x_2x_3^{2}x_4^{4}x_5^{2}$ &  $233.\ x_1x_2^{2}x_3x_4^{2}x_5^{4}$ &  $234.\ x_1x_2^{2}x_3x_4^{4}x_5^{2}$\cr  $235.\ x_1x_2^{2}x_3^{4}x_4x_5^{2}$ & 
\end{tabular}}

\smallskip
$B_5^+(2,4)$ is the set of 5 monomials: $a_{10,t},\ 236 \leqslant t \leqslant 240$.

\medskip
\centerline{\begin{tabular}{llll}
$236.\ x_1x_2^{2}x_3^{2}x_4^{2}x_5^{3}$ &  $237.\ x_1x_2^{2}x_3^{2}x_4^{3}x_5^{2}$ &  $238.\ x_1x_2^{2}x_3^{3}x_4^{2}x_5^{2}$ &  $239.\ x_1x_2^{3}x_3^{2}x_4^{2}x_5^{2}$\cr  
$240.\ x_1^{3}x_2x_3^{2}x_4^{2}x_5^{2}$ &   
\end{tabular}}

\smallskip
$B_5^+(4,1,1)$ is the set of 20 monomials: $a_{10,t},\ 241 \leqslant t \leqslant 260$.

\medskip
\centerline{\begin{tabular}{llll}
$241.\ x_1x_2x_3x_4x_5^{6}$ &  $242.\ x_1x_2x_3x_4^{6}x_5$ &  $243.\ x_1x_2x_3^{6}x_4x_5$ &  $244.\ x_1x_2^{6}x_3x_4x_5$\cr  
$245.\ x_1x_2x_3x_4^{2}x_5^{5}$ &  $246.\ x_1x_2x_3^{2}x_4x_5^{5}$ &  $247.\ x_1x_2x_3^{2}x_4^{5}x_5$ &  $248.\ x_1x_2^{2}x_3x_4x_5^{5}$\cr  
$249.\ x_1x_2^{2}x_3x_4^{5}x_5$ &  $250.\ x_1x_2^{2}x_3^{5}x_4x_5$ &  $251.\ x_1x_2x_3x_4^{3}x_5^{4}$ &  $252.\ x_1x_2x_3^{3}x_4x_5^{4}$\cr  
$253.\ x_1x_2x_3^{3}x_4^{4}x_5$ &  $254.\ x_1x_2^{3}x_3x_4x_5^{4}$ &  $255.\ x_1x_2^{3}x_3x_4^{4}x_5$ &  $256.\ x_1x_2^{3}x_3^{4}x_4x_5$\cr  
$257.\ x_1^{3}x_2x_3x_4x_5^{4}$ &  $258.\ x_1^{3}x_2x_3x_4^{4}x_5$ &  $259.\ x_1^{3}x_2x_3^{4}x_4x_5$ &  $260.\ x_1^{3}x_2^{4}x_3x_4x_5$\cr  
\end{tabular}}

\smallskip
$B_5^+(4,3)$ is the set of 20 monomials: $a_{10,t},\ 261 \leqslant t \leqslant 280$.

\medskip
\centerline{\begin{tabular}{llll}
$261.\ x_1x_2x_3^{2}x_4^{3}x_5^{3}$ &  $262.\ x_1x_2x_3^{3}x_4^{2}x_5^{3}$ &  $263.\ x_1x_2x_3^{3}x_4^{3}x_5^{2}$ &  $264.\ x_1x_2^{2}x_3x_4^{3}x_5^{3}$\cr  
$265.\ x_1x_2^{2}x_3^{3}x_4x_5^{3}$ &  $266.\ x_1x_2^{2}x_3^{3}x_4^{3}x_5$ &  $267.\ x_1x_2^{3}x_3x_4^{2}x_5^{3}$ &  $268.\ x_1x_2^{3}x_3x_4^{3}x_5^{2}$\cr  
$269.\ x_1x_2^{3}x_3^{2}x_4x_5^{3}$ &  $270.\ x_1x_2^{3}x_3^{2}x_4^{3}x_5$ &  $271.\ x_1x_2^{3}x_3^{3}x_4x_5^{2}$ &  $272.\ x_1x_2^{3}x_3^{3}x_4^{2}x_5$\cr  
$273.\ x_1^{3}x_2x_3x_4^{2}x_5^{3}$ &  $274.\ x_1^{3}x_2x_3x_4^{3}x_5^{2}$ &  $275.\ x_1^{3}x_2x_3^{2}x_4x_5^{3}$ &  $276.\ x_1^{3}x_2x_3^{2}x_4^{3}x_5$\cr  
$277.\ x_1^{3}x_2x_3^{3}x_4x_5^{2}$ &  $278.\ x_1^{3}x_2x_3^{3}x_4^{2}x_5$ &  $279.\ x_1^{3}x_2^{3}x_3x_4x_5^{2}$ &  $280.\ x_1^{3}x_2^{3}x_3x_4^{2}x_5$\cr  
\end{tabular}}
\medskip

By a simple computation using the results in \cite{su2}, we have
\begin{align*}B_5(2,2,1) &= \Phi(B_4(2,2,1)),\\
B_5(4,1,1) &= \Phi(B_4(4,1,1))\cup \{x_1^{3}x_2^{4}x_3x_4x_5\},\\
B_5(4,3) &= \Phi(B_4(4,3)). 
\end{align*}

Combining this and the fact that $B_4(2,4) = \emptyset$ we see that Conjecture \ref{gtom} is true for $n = 5$ and the degree 10.

\medskip
\subsection{The admissible monomials of degree 11 in $P_5$}\label{s60}\

\medskip
$B_5^0(11)$ is the set of 240 monomials:

\medskip
\centerline{\begin{tabular}{llll} 
$1.\ x_3x_4^{3}x_5^{7}$ &  $2.\ x_3x_4^{7}x_5^{3}$ &  $3.\ x_3^{3}x_4x_5^{7}$ &  $4.\ x_3^{3}x_4^{7}x_5$\cr  
$5.\ x_3^{7}x_4x_5^{3}$ &  $6.\ x_3^{7}x_4^{3}x_5$ &  $7.\ x_2x_4^{3}x_5^{7}$ &  $8.\ x_2x_4^{7}x_5^{3}$\cr  
$9.\ x_2x_3^{3}x_5^{7}$ &  $10.\ x_2x_3^{3}x_4^{7}$ &  $11.\ x_2x_3^{7}x_5^{3}$ &  $12.\ x_2x_3^{7}x_4^{3}$\cr 
$13.\ x_2^{3}x_4x_5^{7}$ &  $14.\ x_2^{3}x_4^{7}x_5$ &  $15.\ x_2^{3}x_3x_5^{7}$ &  $16.\ x_2^{3}x_3x_4^{7}$\cr  
$17.\ x_2^{3}x_3^{7}x_5$ &  $18.\ x_2^{3}x_3^{7}x_4$ &  $19.\ x_2^{7}x_4x_5^{3}$ &  $20.\ x_2^{7}x_4^{3}x_5$\cr  
$21.\ x_2^{7}x_3x_5^{3}$ &  $22.\ x_2^{7}x_3x_4^{3}$ &  $23.\ x_2^{7}x_3^{3}x_5$ &  $24.\ x_2^{7}x_3^{3}x_4$\cr  
$25.\ x_1x_4^{3}x_5^{7}$ &  $26.\ x_1x_4^{7}x_5^{3}$ &  $27.\ x_1x_3^{3}x_5^{7}$ &  $28.\ x_1x_3^{3}x_4^{7}$\cr  
$29.\ x_1x_3^{7}x_5^{3}$ &  $30.\ x_1x_3^{7}x_4^{3}$ &  $31.\ x_1x_2^{3}x_5^{7}$ &  $32.\ x_1x_2^{3}x_4^{7}$\cr
\end{tabular}}
\centerline{\begin{tabular}{llll}
$149.\ x_1x_2^{2}x_3x_5^{6}$ &  $150.\ x_1x_2^{2}x_3x_4^{6}$ &  $151.\ x_2x_3^{2}x_4^{5}x_5^{2}$ &  $152.\ x_1x_3^{2}x_4^{5}x_5^{2}$\cr  
$153.\ x_1x_2^{2}x_4^{5}x_5^{2}$ &  $154.\ x_1x_2^{2}x_3^{5}x_5^{2}$ &  $155.\ x_1x_2^{2}x_3^{5}x_4^{2}$ &  $156.\ x_2x_3^{3}x_4^{4}x_5^{2}$\cr  
$157.\ x_1x_3^{3}x_4^{4}x_5^{2}$ &  $158.\ x_1x_2^{3}x_4^{4}x_5^{2}$ &  $159.\ x_1x_2^{3}x_3^{4}x_5^{2}$ &  $160.\ x_1x_2^{3}x_3^{4}x_4^{2}$\cr  
$161.\ x_2x_3^{2}x_4^{3}x_5^{5}$ &  $162.\ x_2x_3^{2}x_4^{5}x_5^{3}$ &  $163.\ x_2x_3^{3}x_4^{2}x_5^{5}$ &  $164.\ x_2x_3^{3}x_4^{5}x_5^{2}$\cr  
$165.\ x_2^{3}x_3x_4^{2}x_5^{5}$ &  $166.\ x_2^{3}x_3x_4^{5}x_5^{2}$ &  $167.\ x_2^{3}x_3^{5}x_4x_5^{2}$ &  $168.\ x_2^{3}x_3^{5}x_4^{2}x_5$\cr  
$169.\ x_1x_3^{2}x_4^{3}x_5^{5}$ &  $170.\ x_1x_3^{2}x_4^{5}x_5^{3}$ &  $171.\ x_1x_3^{3}x_4^{2}x_5^{5}$ &  $172.\ x_1x_3^{3}x_4^{5}x_5^{2}$\cr  
$173.\ x_1x_2^{2}x_4^{3}x_5^{5}$ &  $174.\ x_1x_2^{2}x_4^{5}x_5^{3}$ &  $175.\ x_1x_2^{2}x_3^{3}x_5^{5}$ &  $176.\ x_1x_2^{2}x_3^{3}x_4^{5}$\cr  
$177.\ x_1x_2^{2}x_3^{5}x_5^{3}$ &  $178.\ x_1x_2^{2}x_3^{5}x_4^{3}$ &  $179.\ x_1x_2^{3}x_4^{2}x_5^{5}$ &  $180.\ x_1x_2^{3}x_4^{5}x_5^{2}$\cr
\end{tabular}}

\medskip
We have $B_5^+(11) = B_5^+(3,2,1)\cup B_5(3,4)\cup \psi(B_5(3))$.

\smallskip
$B_5^+(3,2,1)$ is the set of 40 monomials:

\medskip
\centerline{\begin{tabular}{llll}
$1.\ x_1x_2x_3x_4^{2}x_5^{6}$ &  $2.\ x_1x_2x_3x_4^{6}x_5^{2}$ &  $3.\ x_1x_2x_3^{2}x_4x_5^{6}$ &  $4.\ x_1x_2x_3^{2}x_4^{2}x_5^{5}$\cr  
$5.\ x_1x_2x_3^{2}x_4^{3}x_5^{4}$ &  $6.\ x_1x_2x_3^{2}x_4^{4}x_5^{3}$ &  $7.\ x_1x_2x_3^{2}x_4^{5}x_5^{2}$ &  $8.\ x_1x_2x_3^{2}x_4^{6}x_5$\cr  
$9.\ x_1x_2x_3^{3}x_4^{2}x_5^{4}$ &  $10.\ x_1x_2x_3^{3}x_4^{4}x_5^{2}$ &  $11.\ x_1x_2x_3^{6}x_4x_5^{2}$ &  $12.\ x_1x_2x_3^{6}x_4^{2}x_5$\cr  
$13.\ x_1x_2^{2}x_3x_4x_5^{6}$ &  $14.\ x_1x_2^{2}x_3x_4^{2}x_5^{5}$ &  $15.\ x_1x_2^{2}x_3x_4^{3}x_5^{4}$ &  $16.\ x_1x_2^{2}x_3x_4^{4}x_5^{3}$\cr  
$17.\ x_1x_2^{2}x_3x_4^{5}x_5^{2}$ &  $18.\ x_1x_2^{2}x_3x_4^{6}x_5$ &  $19.\ x_1x_2^{2}x_3^{3}x_4x_5^{4}$ &  $20.\ x_1x_2^{2}x_3^{3}x_4^{4}x_5$\cr  
$21.\ x_1x_2^{2}x_3^{4}x_4x_5^{3}$ &  $22.\ x_1x_2^{2}x_3^{4}x_4^{3}x_5$ &  $23.\ x_1x_2^{2}x_3^{5}x_4x_5^{2}$ &  $24.\ x_1x_2^{2}x_3^{5}x_4^{2}x_5$\cr  
$25.\ x_1x_2^{3}x_3x_4^{2}x_5^{4}$ &  $26.\ x_1x_2^{3}x_3x_4^{4}x_5^{2}$ &  $27.\ x_1x_2^{3}x_3^{2}x_4x_5^{4}$ &  $28.\ x_1x_2^{3}x_3^{2}x_4^{4}x_5$\cr  
$29.\ x_1x_2^{3}x_3^{4}x_4x_5^{2}$ &  $30.\ x_1x_2^{3}x_3^{4}x_4^{2}x_5$ &  $31.\ x_1x_2^{6}x_3x_4x_5^{2}$ &  $32.\ x_1x_2^{6}x_3x_4^{2}x_5$\cr  
$33.\ x_1^{3}x_2x_3x_4^{2}x_5^{4}$ &  $34.\ x_1^{3}x_2x_3x_4^{4}x_5^{2}$ &  $35.\ x_1^{3}x_2x_3^{2}x_4x_5^{4}$ &  $36.\ x_1^{3}x_2x_3^{2}x_4^{4}x_5$\cr  
$37.\ x_1^{3}x_2x_3^{4}x_4x_5^{2}$ &  $38.\ x_1^{3}x_2x_3^{4}x_4^{2}x_5$ &  $39.\ x_1^{3}x_2^{4}x_3x_4x_5^{2}$ &  $40.\ x_1^{3}x_2^{4}x_3x_4^{2}x_5$\cr 
\end{tabular}}
\medskip

\smallskip
$B_5^+(3,4)$ is the set of 10 monomials:

\medskip
\centerline{\begin{tabular}{llll}
$41.\ x_1x_2^{2}x_3^{2}x_4^{3}x_5^{3}$ &  $42.\ x_1x_2^{2}x_3^{3}x_4^{2}x_5^{3}$ &  $43.\ x_1x_2^{2}x_3^{3}x_4^{3}x_5^{2}$ &  $44.\ x_1x_2^{3}x_3^{2}x_4^{2}x_5^{3}$\cr  
$45.\ x_1x_2^{3}x_3^{2}x_4^{3}x_5^{2}$ &  $46.\ x_1x_2^{3}x_3^{3}x_4^{2}x_5^{2}$ &  $47.\ x_1^{3}x_2x_3^{2}x_4^{2}x_5^{3}$ &  $48.\ x_1^{3}x_2x_3^{2}x_4^{3}x_5^{2}$\cr  
$49.\ x_1^{3}x_2x_3^{3}x_4^{2}x_5^{2}$ &  $50.\ x_1^{3}x_2^{3}x_3x_4^{2}x_5^{2}$ &&\cr
\end{tabular}}

\medskip
By a simple computation using the results in \cite{su2}, we can easily show  that
$B_5(3,2,1) = \Phi(B_4(3,2,1)).$
Combining this and the fact that $B_4(3,4) = \emptyset$ we see that Conjecture \ref{gtom} is true for $n = 5$ and the degree 11.

\subsection{The admissible monomials of degree 25 in $P_5^0$}\label{s62}\

\medskip
We have $B_5(25) = B_5^0(25) \cup B_5^+(25)$.

$B_5^0(25) = B_5^0(3,3,2,1)$ is the set of 520 monomials $a_{25,t},\ 1 \leqslant t \leqslant 520$.

\medskip
\centerline{\begin{tabular}{llll}
$1.\ x_3^{3}x_4^{7}x_5^{15}$ &  $2.\ x_3^{3}x_4^{15}x_5^{7}$ &  $3.\ x_3^{7}x_4^{3}x_5^{15}$ &  $4.\ x_3^{7}x_4^{15}x_5^{3}$\cr  
$5.\ x_3^{15}x_4^{3}x_5^{7}$ &  $6.\ x_3^{15}x_4^{7}x_5^{3}$ &  $7.\ x_2^{3}x_4^{7}x_5^{15}$ &  $8.\ x_2^{3}x_4^{15}x_5^{7}$\cr  
$9.\ x_2^{3}x_3^{7}x_5^{15}$ &  $10.\ x_2^{3}x_3^{7}x_4^{15}$ &  $11.\ x_2^{3}x_3^{15}x_5^{7}$ &  $12.\ x_2^{3}x_3^{15}x_4^{7}$\cr  
$13.\ x_2^{7}x_4^{3}x_5^{15}$ &  $14.\ x_2^{7}x_4^{15}x_5^{3}$ &  $15.\ x_2^{7}x_3^{3}x_5^{15}$ &  $16.\ x_2^{7}x_3^{3}x_4^{15}$\cr  
$17.\ x_2^{7}x_3^{15}x_5^{3}$ &  $18.\ x_2^{7}x_3^{15}x_4^{3}$ &  $19.\ x_2^{15}x_4^{3}x_5^{7}$ &  $20.\ x_2^{15}x_4^{7}x_5^{3}$\cr  
$21.\ x_2^{15}x_3^{3}x_5^{7}$ &  $22.\ x_2^{15}x_3^{3}x_4^{7}$ &  $23.\ x_2^{15}x_3^{7}x_5^{3}$ &  $24.\ x_2^{15}x_3^{7}x_4^{3}$\cr  
$25.\ x_1^{3}x_4^{7}x_5^{15}$ &  $26.\ x_1^{3}x_4^{15}x_5^{7}$ &  $27.\ x_1^{3}x_3^{7}x_5^{15}$ &  $28.\ x_1^{3}x_3^{7}x_4^{15}$\cr  
$29.\ x_1^{3}x_3^{15}x_5^{7}$ &  $30.\ x_1^{3}x_3^{15}x_4^{7}$ &  $31.\ x_1^{3}x_2^{7}x_5^{15}$ &  $32.\ x_1^{3}x_2^{7}x_4^{15}$\cr
\end{tabular}}
\centerline{\begin{tabular}{llll}
$193.\ x_1^{3}x_2^{5}x_3^{2}x_4^{15}$ &  $194.\ x_1^{3}x_2^{5}x_3^{15}x_5^{2}$ &  $195.\ x_1^{3}x_2^{5}x_3^{15}x_4^{2}$ &  $196.\ x_1^{3}x_2^{15}x_4x_5^{6}$\cr  
$197.\ x_1^{3}x_2^{15}x_4^{3}x_5^{4}$ &  $198.\ x_1^{3}x_2^{15}x_4^{5}x_5^{2}$ &  $199.\ x_1^{3}x_2^{15}x_3x_5^{6}$ &  $200.\ x_1^{3}x_2^{15}x_3x_4^{6}$\cr  
$201.\ x_1^{3}x_2^{15}x_3^{3}x_5^{4}$ &  $202.\ x_1^{3}x_2^{15}x_3^{3}x_4^{4}$ &  $203.\ x_1^{3}x_2^{15}x_3^{5}x_5^{2}$ &  $204.\ x_1^{3}x_2^{15}x_3^{5}x_4^{2}$\cr  
$205.\ x_1^{15}x_3x_4^{6}x_5^{3}$ &  $206.\ x_1^{15}x_3^{3}x_4x_5^{6}$ &  $207.\ x_1^{15}x_3^{3}x_4^{3}x_5^{4}$ &  $208.\ x_1^{15}x_3^{3}x_4^{5}x_5^{2}$\cr  
$209.\ x_1^{15}x_2x_4^{6}x_5^{3}$ &  $210.\ x_1^{15}x_2x_3^{6}x_5^{3}$ &  $211.\ x_1^{15}x_2x_3^{6}x_4^{3}$ &  $212.\ x_1^{15}x_2^{3}x_4x_5^{6}$\cr  
$213.\ x_1^{15}x_2^{3}x_4^{3}x_5^{4}$ &  $214.\ x_1^{15}x_2^{3}x_4^{5}x_5^{2}$ &  $215.\ x_1^{15}x_2^{3}x_3x_5^{6}$ &  $216.\ x_1^{15}x_2^{3}x_3x_4^{6}$\cr  
$217.\ x_1^{15}x_2^{3}x_3^{3}x_5^{4}$ &  $218.\ x_1^{15}x_2^{3}x_3^{3}x_4^{4}$ &  $219.\ x_1^{15}x_2^{3}x_3^{5}x_5^{2}$ &  $220.\ x_1^{15}x_2^{3}x_3^{5}x_4^{2}$\cr  
$221.\ x_2x_3^{3}x_4^{6}x_5^{15}$ &  $222.\ x_2x_3^{3}x_4^{15}x_5^{6}$ &  $223.\ x_2x_3^{15}x_4^{3}x_5^{6}$ &  $224.\ x_2^{15}x_3x_4^{3}x_5^{6}$\cr
\end{tabular}}
\centerline{\begin{tabular}{llll}
$385.\ x_1^{7}x_3x_4^{10}x_5^{7}$ &  $386.\ x_1^{7}x_3^{3}x_4^{5}x_5^{10}$ &  $387.\ x_1^{7}x_3^{3}x_4^{9}x_5^{6}$ &  $388.\ x_1^{7}x_3^{9}x_4^{3}x_5^{6}$\cr  
$389.\ x_1^{7}x_2x_4^{10}x_5^{7}$ &  $390.\ x_1^{7}x_2x_3^{10}x_5^{7}$ &  $391.\ x_1^{7}x_2x_3^{10}x_4^{7}$ &  $392.\ x_1^{7}x_2^{3}x_4^{5}x_5^{10}$\cr  
$393.\ x_1^{7}x_2^{3}x_4^{9}x_5^{6}$ &  $394.\ x_1^{7}x_2^{3}x_3^{5}x_5^{10}$ &  $395.\ x_1^{7}x_2^{3}x_3^{5}x_4^{10}$ &  $396.\ x_1^{7}x_2^{3}x_3^{9}x_5^{6}$\cr  
$397.\ x_1^{7}x_2^{3}x_3^{9}x_4^{6}$ &  $398.\ x_1^{7}x_2^{9}x_4^{3}x_5^{6}$ &  $399.\ x_1^{7}x_2^{9}x_3^{3}x_5^{6}$ &  $400.\ x_1^{7}x_2^{9}x_3^{3}x_4^{6}$\cr  
$401.\ x_2x_3^{7}x_4^{7}x_5^{10}$ &  $402.\ x_2^{3}x_3x_4^{14}x_5^{7}$ &  $403.\ x_2^{3}x_3^{7}x_4^{5}x_5^{10}$ &  $404.\ x_2^{3}x_3^{13}x_4^{3}x_5^{6}$\cr  
$405.\ x_2^{7}x_3x_4^{7}x_5^{10}$ &  $406.\ x_1x_3^{7}x_4^{7}x_5^{10}$ &  $407.\ x_1x_2^{7}x_4^{7}x_5^{10}$ &  $408.\ x_1x_2^{7}x_3^{7}x_5^{10}$\cr  
$409.\ x_1x_2^{7}x_3^{7}x_4^{10}$ &  $410.\ x_1^{3}x_3x_4^{14}x_5^{7}$ &  $411.\ x_1^{3}x_3^{7}x_4^{5}x_5^{10}$ &  $412.\ x_1^{3}x_3^{13}x_4^{3}x_5^{6}$\cr  
$413.\ x_1^{3}x_2x_4^{14}x_5^{7}$ &  $414.\ x_1^{3}x_2x_3^{14}x_5^{7}$ &  $415.\ x_1^{3}x_2x_3^{14}x_4^{7}$ &  $416.\ x_1^{3}x_2^{7}x_4^{5}x_5^{10}$\cr
\end{tabular}}

\subsection{The admissible monomials of degree 25 in $P_5^+$}\label{s63}\

\medskip
We have $B_5^+(25) = B_5^+(3,3,2,1)\cup \psi(B_5(10))$, where $B_5^+(3,3,2,1)$ is the set of 440 monomials $b_{25,u},\ 1 \leqslant u \leqslant 440$:

\centerline{\begin{tabular}{llll}
$1.\ x_1x_2x_3^{2}x_4^{6}x_5^{15}$ &  $2.\ x_1x_2x_3^{2}x_4^{15}x_5^{6}$ &  $3.\ x_1x_2x_3^{6}x_4^{2}x_5^{15}$ &  $4.\ x_1x_2x_3^{6}x_4^{15}x_5^{2}$\cr  
$5.\ x_1x_2x_3^{15}x_4^{2}x_5^{6}$ &  $6.\ x_1x_2x_3^{15}x_4^{6}x_5^{2}$ &  $7.\ x_1x_2^{2}x_3^{3}x_4^{4}x_5^{15}$ &  $8.\ x_1x_2^{2}x_3^{3}x_4^{15}x_5^{4}$\cr  
$9.\ x_1x_2^{2}x_3^{4}x_4^{3}x_5^{15}$ &  $10.\ x_1x_2^{2}x_3^{4}x_4^{15}x_5^{3}$ &  $11.\ x_1x_2^{2}x_3^{15}x_4^{3}x_5^{4}$ &  $12.\ x_1x_2^{2}x_3^{15}x_4^{4}x_5^{3}$\cr  
$13.\ x_1x_2^{3}x_3^{2}x_4^{4}x_5^{15}$ &  $14.\ x_1x_2^{3}x_3^{2}x_4^{15}x_5^{4}$ &  $15.\ x_1x_2^{3}x_3^{15}x_4^{2}x_5^{4}$ &  $16.\ x_1x_2^{6}x_3x_4^{2}x_5^{15}$\cr  
$17.\ x_1x_2^{6}x_3x_4^{15}x_5^{2}$ &  $18.\ x_1x_2^{6}x_3^{15}x_4x_5^{2}$ &  $19.\ x_1x_2^{15}x_3x_4^{2}x_5^{6}$ &  $20.\ x_1x_2^{15}x_3x_4^{6}x_5^{2}$\cr  
$21.\ x_1x_2^{15}x_3^{2}x_4^{3}x_5^{4}$ &  $22.\ x_1x_2^{15}x_3^{2}x_4^{4}x_5^{3}$ &  $23.\ x_1x_2^{15}x_3^{3}x_4^{2}x_5^{4}$ &  $24.\ x_1x_2^{15}x_3^{6}x_4x_5^{2}$\cr  
$25.\ x_1^{3}x_2x_3^{2}x_4^{4}x_5^{15}$ &  $26.\ x_1^{3}x_2x_3^{2}x_4^{15}x_5^{4}$ &  $27.\ x_1^{3}x_2x_3^{4}x_4^{2}x_5^{15}$ &  $28.\ x_1^{3}x_2x_3^{4}x_4^{15}x_5^{2}$\cr  
$29.\ x_1^{3}x_2x_3^{15}x_4^{2}x_5^{4}$ &  $30.\ x_1^{3}x_2x_3^{15}x_4^{4}x_5^{2}$ &  $31.\ x_1^{3}x_2^{4}x_3x_4^{2}x_5^{15}$ &  $32.\ x_1^{3}x_2^{4}x_3x_4^{15}x_5^{2}$\cr  
$33.\ x_1^{3}x_2^{4}x_3^{15}x_4x_5^{2}$ &  $34.\ x_1^{3}x_2^{15}x_3x_4^{2}x_5^{4}$ &  $35.\ x_1^{3}x_2^{15}x_3x_4^{4}x_5^{2}$ &  $36.\ x_1^{3}x_2^{15}x_3^{4}x_4x_5^{2}$\cr  
$37.\ x_1^{15}x_2x_3x_4^{2}x_5^{6}$ &  $38.\ x_1^{15}x_2x_3x_4^{6}x_5^{2}$ &  $39.\ x_1^{15}x_2x_3^{2}x_4^{3}x_5^{4}$ &  $40.\ x_1^{15}x_2x_3^{2}x_4^{4}x_5^{3}$\cr  
\end{tabular}}
\centerline{\begin{tabular}{llll}
$41.\ x_1^{15}x_2x_3^{3}x_4^{2}x_5^{4}$ &  $42.\ x_1^{15}x_2x_3^{6}x_4x_5^{2}$ &  $43.\ x_1^{15}x_2^{3}x_3x_4^{2}x_5^{4}$ &  $44.\ x_1^{15}x_2^{3}x_3x_4^{4}x_5^{2}$\cr  
$45.\ x_1^{15}x_2^{3}x_3^{4}x_4x_5^{2}$ &  $46.\ x_1x_2^{2}x_3x_4^{6}x_5^{15}$ &  $47.\ x_1x_2^{2}x_3x_4^{15}x_5^{6}$ &  $48.\ x_1x_2^{2}x_3^{5}x_4^{2}x_5^{15}$\cr  
$49.\ x_1x_2^{2}x_3^{5}x_4^{15}x_5^{2}$ &  $50.\ x_1x_2^{2}x_3^{15}x_4x_5^{6}$ &  $51.\ x_1x_2^{2}x_3^{15}x_4^{5}x_5^{2}$ &  $52.\ x_1x_2^{3}x_3^{4}x_4^{2}x_5^{15}$\cr  
$53.\ x_1x_2^{3}x_3^{4}x_4^{15}x_5^{2}$ &  $54.\ x_1x_2^{3}x_3^{15}x_4^{4}x_5^{2}$ &  $55.\ x_1x_2^{15}x_3^{2}x_4x_5^{6}$ &  $56.\ x_1x_2^{15}x_3^{2}x_4^{5}x_5^{2}$\cr  
$57.\ x_1x_2^{15}x_3^{3}x_4^{4}x_5^{2}$ &  $58.\ x_1^{15}x_2x_3^{2}x_4x_5^{6}$ &  $59.\ x_1^{15}x_2x_3^{2}x_4^{5}x_5^{2}$ &  $60.\ x_1^{15}x_2x_3^{3}x_4^{4}x_5^{2}$\cr  
$61.\ x_1x_2x_3^{2}x_4^{7}x_5^{14}$ &  $62.\ x_1x_2x_3^{2}x_4^{14}x_5^{7}$ &  $63.\ x_1x_2x_3^{7}x_4^{2}x_5^{14}$ &  $64.\ x_1x_2x_3^{7}x_4^{14}x_5^{2}$\cr  
$65.\ x_1x_2x_3^{14}x_4^{2}x_5^{7}$ &  $66.\ x_1x_2x_3^{14}x_4^{7}x_5^{2}$ &  $67.\ x_1x_2^{2}x_3^{4}x_4^{7}x_5^{11}$ &  $68.\ x_1x_2^{2}x_3^{4}x_4^{11}x_5^{7}$\cr  
$69.\ x_1x_2^{2}x_3^{7}x_4^{7}x_5^{8}$ &  $70.\ x_1x_2^{2}x_3^{7}x_4^{8}x_5^{7}$ &  $71.\ x_1x_2^{3}x_3^{3}x_4^{4}x_5^{14}$ &  $72.\ x_1x_2^{3}x_3^{5}x_4^{2}x_5^{14}$\cr
\end{tabular}}
\centerline{\begin{tabular}{llll}
$233.\ x_1x_2^{7}x_3^{6}x_4^{9}x_5^{2}$ &  $234.\ x_1x_2^{7}x_3^{9}x_4^{6}x_5^{2}$ &  $235.\ x_1^{3}x_2x_3^{3}x_4^{4}x_5^{14}$ &  $236.\ x_1^{3}x_2x_3^{3}x_4^{14}x_5^{4}$\cr  
$237.\ x_1^{3}x_2x_3^{7}x_4^{6}x_5^{8}$ &  $238.\ x_1^{3}x_2x_3^{7}x_4^{10}x_5^{4}$ &  $239.\ x_1^{3}x_2x_3^{13}x_4^{6}x_5^{2}$ &  $240.\ x_1^{3}x_2^{3}x_3^{4}x_4x_5^{14}$\cr  
$241.\ x_1^{3}x_2^{3}x_3^{4}x_4^{13}x_5^{2}$ &  $242.\ x_1^{3}x_2^{3}x_3^{13}x_4^{4}x_5^{2}$ &  $243.\ x_1^{3}x_2^{7}x_3x_4^{4}x_5^{10}$ &  $244.\ x_1^{3}x_2^{7}x_3x_4^{10}x_5^{4}$\cr  
$245.\ x_1^{3}x_2^{7}x_3^{3}x_4^{4}x_5^{8}$ &  $246.\ x_1^{3}x_2^{7}x_3^{3}x_4^{8}x_5^{4}$ &  $247.\ x_1^{3}x_2^{7}x_3^{4}x_4x_5^{10}$ &  $248.\ x_1^{3}x_2^{7}x_3^{4}x_4^{9}x_5^{2}$\cr  
$249.\ x_1^{7}x_2x_3^{3}x_4^{10}x_5^{4}$ &  $250.\ x_1^{7}x_2x_3^{6}x_4x_5^{10}$ &  $251.\ x_1^{7}x_2x_3^{6}x_4^{9}x_5^{2}$ &  $252.\ x_1^{7}x_2x_3^{9}x_4^{6}x_5^{2}$\cr  
$253.\ x_1^{7}x_2^{3}x_3x_4^{4}x_5^{10}$ &  $254.\ x_1^{7}x_2^{3}x_3x_4^{10}x_5^{4}$ &  $255.\ x_1^{7}x_2^{3}x_3^{4}x_4x_5^{10}$ &  $256.\ x_1^{7}x_2^{3}x_3^{4}x_4^{9}x_5^{2}$\cr  
$257.\ x_1^{7}x_2^{3}x_3^{9}x_4^{2}x_5^{4}$ &  $258.\ x_1^{7}x_2^{3}x_3^{9}x_4^{4}x_5^{2}$ &  $259.\ x_1x_2x_3^{7}x_4^{10}x_5^{6}$ &  $260.\ x_1x_2^{6}x_3x_4^{11}x_5^{6}$\cr  
$261.\ x_1^{3}x_2^{3}x_3^{3}x_4^{12}x_5^{4}$ &  $262.\ x_1^{3}x_2^{3}x_3^{7}x_4^{8}x_5^{4}$ &  $263.\ x_1x_2^{3}x_3^{3}x_4^{14}x_5^{4}$ &  $264.\ x_1x_2^{3}x_3^{7}x_4^{6}x_5^{8}$\cr
\end{tabular}}
\centerline{\begin{tabular}{llll}
$425.\ x_1^{7}x_2x_3^{11}x_4^{4}x_5^{2}$ &  $426.\ x_1^{7}x_2x_3^{14}x_4x_5^{2}$ &  $427.\ x_1^{7}x_2^{3}x_3x_4^{6}x_5^{8}$ &  $428.\ x_1^{7}x_2^{3}x_3x_4^{8}x_5^{6}$\cr  
$429.\ x_1^{7}x_2^{3}x_3^{5}x_4^{2}x_5^{8}$ &  $430.\ x_1^{7}x_2^{3}x_3^{5}x_4^{8}x_5^{2}$ &  $431.\ x_1^{7}x_2^{3}x_3^{8}x_4x_5^{6}$ &  $432.\ x_1^{7}x_2^{3}x_3^{8}x_4^{5}x_5^{2}$\cr  
$433.\ x_1^{7}x_2^{8}x_3x_4^{3}x_5^{6}$ &  $434.\ x_1^{7}x_2^{8}x_3^{3}x_4x_5^{6}$ &  $435.\ x_1^{7}x_2^{8}x_3^{3}x_4^{5}x_5^{2}$ &  $436.\ x_1^{7}x_2^{9}x_3x_4^{6}x_5^{2}$\cr  
$437.\ x_1^{7}x_2^{9}x_3^{2}x_4^{3}x_5^{4}$ &  $438.\ x_1^{7}x_2^{11}x_3x_4^{2}x_5^{4}$ &  $439.\ x_1^{7}x_2^{11}x_3x_4^{4}x_5^{2}$ &  $440.\ x_1^{7}x_2^{11}x_3^{4}x_4x_5^{2}$\cr  
\end{tabular}}

\medskip
By a direct computation using the results in \cite{su2}, we have 
$$|\Phi(B_4(3,3,2,1))| = 361 \text{ and } \Phi(B_4(3,3,2,1)) \subset B_5(3,3,2,1).$$ 
If $\omega$ is a weight vector of degree 25 and $\omega \ne (3,3,2,1)$, then $B_4(\omega) = \emptyset$. If $s > 1$, then $B_4(15.2^s-5) = \emptyset$. Hence, Conjecture \ref{gtom} is true for $n = 5$ and the degree $15.2^s-5$ with $s$ an arbitrary positive integer.

\section*{Acknowledgment} 

The first manuscript of this paper was written when the author was visiting the Vietnam Institute for Advanced Study in Mathematics (VIASM) from August to November 2017. He would like to thank the VIASM for financial support, convenient working condition and for kind hospitality.

The author was supported in part by the National Foundation for Science and Technology Development (NAFOSTED) of Viet Nam under the grant number 101.04-2017.05.


\bibliographystyle{amsplain}

\end{document}